\documentclass[11pt]{amsart}

\usepackage[colorlinks,linkcolor={blue}]{hyperref}
\usepackage[pagewise]{lineno}%\linenumbers

 \usepackage{graphicx} 
 \usepackage{amscd}
 \usepackage{enumerate}

\usepackage[colorinlistoftodos]{todonotes}
\makeatletter
\providecommand\@dotsep{5}
\makeatother
 \def\a{\alpha}
 
 \def\ha{{\widehat\a}}
 \def\be{\beta}
 
 \def\de{\delta}
 
 \def\e{\varepsilon}
 
 \def\heta{{\widehat\eta}}

 \def\ga{\gamma}
 \def\dga{{\dot{\gamma}}}
 \def\hga{{\widehat{\gamma}}}
 
 \def\Ga{\Gamma}
 \def\hGa{{\widehat\Gamma}}
 
 \def\vr{\varphi}

 \def\la{\lambda}
 
 \def\La{\Lambda}
 \def\si{\sigma}
 
 \def\om{\omega}
 \def\Om{\Omega}
 \def\whom{{\widehat{\om}}}

 \def\th{\theta}

 \def\re{{\mathbb R}}
 \def\na{{\mathbb N}}

 \def\then{\Longrightarrow}
 \def\ov{\overline}
 \def\Z{{\mathbb Z}}

 \def\D{{\mathbb D}}
 \def\hd{{\hat d}}
 
 \def\hE{{\widehat E}}

 \def\H{{\mathbb H}}

 \def\hH{{\widehat H}}

 \def\K{{\mathbb K}}

 \def\cL{{\mathcal L}}
 
 \def\hL{{\widehat{L}}}

 \def\M{{\mathbb M}}

 \def\hM{{\widehat{M}}}
 \def\tM{{\widetilde{M}}}

 \def\on{{\ov{n}}}

 \def\SS{{\mathbb S}}
 \def\T{{\mathbb T}}

 \def\tq1{{\tilde{q}_1}}

 \def\hw{{\widehat{w}}}

 \def\oH{{\ov{H}}}

 \def \lv{\left\vert}
 \def \rv{\right\vert}
 \def \lV{\left\Vert}
 \def \rV{\right\Vert}
 \def \ov{\overline}
 
 \def \then{\Longrightarrow}

 \DeclareMathOperator*{\tsum}{{\textstyle \sum}}

 \DeclareMathOperator{\diam}{diam}
 \DeclareMathOperator{\length}{length}
 
 \DeclareMathOperator{\interior}{int}

 \DeclareMathOperator{\Lip}{Lip}

  \renewcommand{\proofname}{{\bf Proof:}}

 \theoremstyle{plain}

  \swapnumbers

 \newtheorem{Thm}{Theorem}[section]
 \newtheorem{Prop}[Thm]{\bf Proposition}
 \newtheorem{Lemma}[Thm]{\bf Lemma}
 
 \newtheorem{Corollary}[Thm]{\bf Corollary}
 \newtheorem{Theorem}[Thm]{\bf Theorem}
 \newtheorem{Proposition}[Thm]{\bf Proposition}

\newtheorem{claim}{\bf Claim}
 
 \theoremstyle{definition}

 \theoremstyle{remark}

 \newtheoremstyle{Cl}% name
  {5pt}%      Space above
  {3pt}%      Space below
  {\sl}%   Body font
  {}%         Indent amount (empty = no indent, \parindent = para indent)
  {\it}% Thm head font
  {:}%        Punctuation after thm head
  {.5em}%     Space after thm head: " " = normal interword space;
  {}%         Thm head spec (can be left empty, meaning `normal')

 \theoremstyle{Cl}

 \def\begincproof{
                  \renewcommand{\proofname}{\it Proof:}
                  \begin{proof}
                 }

 \def\endcproof{
                \renewcommand{\qedsymbol}{$\diamondsuit$}
                \end{proof} 
                \renewcommand{\qedsymbol}{\openbox}
                \renewcommand{\proofname}{\bf Proof:}
               }

 \renewcommand{\proofname}{{\bf Proof:}}

 \title[Homogenization on manifolds]
 {Homogenization on arbitrary manifolds\\
 Proof with test functions} 

\author{Gonzalo Contreras}  

\address{Gonzalo Contreras\newline\indent 
Centro de Investigaci\'on en Matem\'aticas\newline\indent 
A.P. 402, 36.000, Guanajuato, GTO, Mexico}
\email{gonzalo@cimat.mx}

\address{}
\author{Yuriria Estrada}
\address{Yuriria Estrada\newline\indent 
Centro de Investigaci\'on en Matem\'aticas\newline\indent 
A.P. 402, 36.000, Guanajuato, GTO, Mexico}
\email{yurira@cimat.mx}
\thanks{Gonzalo Contreras was partially supported by CONACYT, Mexico, grant A1-S-10145.}

\thanks{{\it Key words and phrases:} Aubry-Mather theory, Hamilton-Jacobi equations, Homogenization.} 
\thanks{Corresponding author: Gonzalo Contreras}

\subjclass[2020]{37J51, 70H20, 35B27}

\begin{document}

\begin{abstract}
We proof the homogenization of the Hamilton-Jacobi equation on arbitrary compact manifolds 
using Evans perturbed test function method.
\end{abstract}

\maketitle

\tableofcontents

%\listoftodos{}

\section{Introduction.}
In this paper we proof the homogenization of the Hamilton-Jacobi equation in 
arbitrary manifolds using Evans perturbed test function method.
A setting for generalizations of periodic homogenization results for PDE's in $\re^n$
to arbitrary manifolds was presented in Contreras, Iturriaga, Siconolfi \cite{CIS} together
with a proof of the homogenization of the Hamilton-Jacobi equation on arbitrary manifolds
using the convergence of the formula of the solution. See also Contreras \cite{mca2013}.

\nocite{Concordel}
\nocite{Tran}
\nocite{Bressan}
\nocite{Ev5}

We review the setting of the homogenization of the Hamilton-Jacobi equation on manifolds.
Let $M$ be a compact boundaryless manifold and $T^*M$ its cotangent bundle.
 A Tonelli hamiltonian on $M$ is a $C^2$ function $H:T^*M\to\re$ satisfying
\begin{enumerate}[(a)]
\item  Convexity: 
$\tfrac{\partial H}{\partial p\partial p} (x,p)$ is positive definite  
$\forall (x,p)\ \in T^*M$.
\item Superlinearity:
$\lim\limits_{|p|\to\infty} \frac{H(x,p)}{|p|}=\infty$ uniformly on $x\in M$.
\end{enumerate}

The maximal abelian cover $\hM$ is the covering map  $\hM\to M$ with group of 
Deck transformations $H_1(M,\Z)$. We see $\hM$ as an (infinite) 
repetition of a fundamental domain with a ``periodicity'' given by the group $H_1(M,\Z)$.
On the covering $\hM$ we consider the Hamilton-Jacobi equation for $u:\hM\times \re^+\to\re$:
\begin{align*}
\partial_t u + \hH(x, \partial_xu) = 0,
\\
u(x,0)=f(x);
\end{align*}
where $\hH:T^*\hM\to\re$ is the lift of $H$.
In the homogenization problem we want to see $\hM$ ``from far away'' and prove that the
solution approximates a solution of a simpler equation as we see it from far away.

\subsection{Convergence of spaces.}\quad

To ``see $\hM$ from far away" is interpreted as multiplying the metric of $\hM$ by $\e$ and letting $\e$ tend to zero.  

The following definition is inspired in Gromov's Hausdorff convergence but it is made ad hoc for our homogenization problem. Let $M_\e$ be the metric space $(\hM,d_\e)$ where $d_\e:=\e\,\hd$ and $\hd$ is the metric on $\hM$ obtained by lifting the riemannian metric on $M$.
Since the passage of $\e\to 0$ will only see the large scale properties of the space $M$, 
our definition will be invariant under  quasi-isometries.
 In fact $\hM$ is quasi-isometric to the group of Deck transformations $H_1(M,\Z)$, 
and $M_\e$ is quasi-isometric to $\e H_1(M,\Z)$. We want to make a formal definition for the 
intuitive fact that 
$$
\e \, H_1(M,\Z) = \e \Z_{a_1}\oplus\cdots\oplus \e \Z_{a_k}\oplus \e \Z^n
\xrightarrow{\;\;\e\downarrow 0\;\;} \re^n =H_1(M,\re).
$$

Let $(M,d)$, $(M_n,d_n)_{n\in\na}$ be complete metric spaces and  $F_n:(M_n,d_n)\to (M,d)$ continuous functions.
We say that $\lim_n(M_n,d_n,F_n)=(M,d)$ if 
\begin{enumerate}
\item\label{1a} There are $B, A_n>0$ with $\lim_n A_n=0$ such that
\begin{equation*}\label{ea}
\forall x,y\in M_n:\qquad 
B^{-1} d_n(x,y) - A_n \le d(F_n(x),F_n(y)) \le B\, d_n(x,y)+A_n.
\end{equation*}
\item\label{1b} For all $y\in M$ and $n\in\na$ there are $x_n\in M_n$ such that $\lim_n F_n(x_n)=y$.
\end{enumerate}
Observe that \eqref{1b} is a kind of surjectivity condition. And \eqref{1a} implies that 
$$
\diam F_n^{-1}\{y\} \le B \, A_n \xrightarrow{\;\;n\;\;} 0,
$$
a kind of injectivity condition.
\addtocounter{equation}{2}

If $\lim_n(M_n,d_n,F_n)=(M,d)$, and $f_n:(M_n,d_n)\to\re$, $f:(M,d)\to\re$ are continuous,
we say that $\lim_n f_n=f$ {\it uniformly on compact sets} if for every compact set $K\subset M$
\begin{equation}\label{ucs}
\lim_n\sup_{x\in F_n^{-1}(K)} 
\lv f_n(x) - f(F_n(x))\rv 
=0.
\end{equation}
We say that the famlity $\{f_n\}$ is {\it equi-Lipschitz} if there is $K>0$ such that 
$$
\forall n \qquad \forall x,y\in M_n \qquad
\lv f_n(x)-f_n(y)\rv < K\, d_n(x,y).
$$
\bigskip

Fix $x_0\in\hM$.
Fix a basis $c_1,\ldots c_k$ for $H^1(M,\re)$.
Fix closed 1-forms $\om_i$ on $M$ such that $c_i=[\om_i]$.
Define $G:\hM\to H_1(M,\re)= H^1(M,\re)^*$ by
\begin{equation}\label{dfg}
G(x)\cdot c_i = \oint_{x_0}^x \whom_i\;,
\end{equation}
where $\whom_i$ is the pullback of $\om_i$ on $\hM$.
Let $F_\e:(M_\e,d_\e)\to H_1(M,\re)$ be 
\begin{equation}\label{Fe}
F_\e(x):= \e\, G(x).
\end{equation}
    
\begin{Proposition} [{ \cite[p. 241]{CIS}}]\qquad
$\lim_{\e\to 0}(\hM,\e d, F_\e)= H_1(M,\re)$.
\end{Proposition}    

Here all the torsion of the fundamental group is killed in the limit so it is the same limit for the 
maximal abelian cover as the limit for  the maximal free abelian cover.

In our homogenization result we will have a family of solutions $v^\e$ of properly scaled 
family of Hamilton-Jacobi equations on $M_\e$. 
Then we prove that $v^\e$ has a limit $u$ defined on the limit space 
$\lim_{\e\to 0} M_\e = H_1(M,\re)$, which is a solution of a simpler Hamilton-Jacobi equation.

\subsection{Invariance under quasi-isometries.}\label{sqi}
\quad

Given two metric spaces $(M_1,d_1)$ and $(M_2,d_2)$ a function $q:M_1\to M_2$ is a
{\it quasi-isometry} 
if there are constants $A>1$, $B, C>0$ such that 
\begin{gather}
\forall x,y\in M_1 \qquad 
A^{-1} d_1(x,y)-B\le d_2(q(x),q(y))\le A_1 d(x,y)+B,
\label{2a}
\\
\forall z\in M_2 \qquad \exists x\in M_1 \qquad d_2(z,q(x))\le C.
\label{2b}
\end{gather}
We say that the metric spaces $(M_1,d_1)$, $(M_2,d_2)$ are {\it quasi-isometric} 
if there exists a quasi-isometry $q:(M_1,d_1)\to (M_2,d_2)$.

If $q_1:(M_1,d_1)\to (M_2,d_2)$ is a quasi-isometry, 
then there exists a quasi-isometry $q_2:(M_2,d_2)\to (M_1,d_1)$.
Indeed, given $z\in M_2$, by \eqref{2b} there is $x\in M_1$ such that 
$d(q_1(x),z)\le C$. It is enough to take $q_2(z)=x$. In fact the quasi-isometry
 is an equivalence relation.

\begin{Proposition}\quad

If $(M_1,d_1)$, $(M_2,d_2)$ are quasi-isometric and 
$\lim_\e (M_2,\e\,d_2)$ exists then
$$
\lim_\e(M_1,\e d_1)=\lim_\e(M_2,\e d_2).
$$
\end{Proposition}

\begin{proof}\quad

Suppose that 
$(K,d)=\lim_\e(M_2,\e\,d_2,F_\e)$ and 
$B,A_\e>0$ satisfy \eqref{1a}. Then
\begin{equation}\label{1ineq}
\forall z,w\in M_2\qquad
B^{-1} \, \e\, d_2(z,w) - A_\e \le
d(F_\e (z),F_\e(w))
\le B\,\e\, d_2(z,w)+A_\e.
\end{equation}
Let $q:M_1\to M_2$ be a quasi-isometry and let $A_1>1$, $B_1,C_1>0$ satisfy \eqref{2a} and \eqref{2b}.  
Applying  inequalities~\eqref{1ineq} to $z=q(x)$, $w=q(y)$ we have that 
for all $x,y\in M_1$
\begin{align*}
B^{-1} \, \e\, d_2(q(x),q(y)) - A_\e 
\le
d(F_\e& q(x),F_\e q(y))
\le B\,\e\, d_2(q(x),q(y))+A_\e,
\\
(B A_1)^{-1}\, \e\, d_2(q(x),q(y)) - (A_\e+\e B_1) \le
d(F_\e& q(x),F_\e q(y))
\le
\\
&\le (B A_1)\,\e\, d_2(q(x),q(y))+(A_\e+\e B_1).
\end{align*}
Thus $F_\e\circ q:(M_1,\e d_1)\to (K,d)$ satisfies \eqref{1a}.

Let $w\in K$, by \eqref{1b} for every $\e$ there is $z_\e\in M_2$
such that 
\begin{equation}
\lim_\e d(F_\e(z_\e),w)=0.
\end{equation}
By \eqref{2b} there is $x_\e\in M_1$ such that 
$d_2(q(x_\e),z_\e)\le C$.
Then
\begin{align*}
d(F_\e q(x_\e),w) 
&\le d(F_\e q(x_\e),F_\e (z_\e))+d(F_\e(z_\e),w)
\\
&\le B\,\e\,d_2(q(x_\e),z_\e) + A_\e + d(F_\e(z_\e),w)
\\
&\le B\, \e\, C + A_\e +d(F_\e(z_\e),w)
\xrightarrow{\;\;\e\;\;} 0.
\end{align*}
Therefore  $\lim_\e (M_1,\e\,d_1, F_\e\circ q) = (K,d)$.

\end{proof}

\subsection{The Scaling.} \quad

In the usual homogenization of the Hamilton-Jacobi equation in $\re^n$ the proper
scaling to obtain convergence is just to 
``replace the coefficients $x$ of the PDE by $\tfrac x\e$'', this is
\begin{equation}
\begin{aligned}
\partial_t u^\e (x,t) + H(\tfrac x\e, \partial_x u(x,t)) =0,
 \\
u^\e(x,0 ) = f(x).
\end{aligned}
\label{ehj}
\end{equation}
Our way to interpret this in a manifold is to consider the function 
$$
v^\e(\tfrac x\e,t) := u^\e(x,t).
$$
The equation \eqref{ehj} becomes
\begin{align}
\partial_t v^\e + H(y, \tfrac 1\e \partial_y v^\e) =0,
 \label{ev1hj}
 \\
v^\e(y,0 ) = f(\e y).
\label{ev2hj}
\end{align}
Now equation~\eqref{ev1hj} makes sense in the manifold $M_\e$, with
$v^\e:M_\e\times[0,+\infty[\to\re$.
And equation \eqref{ev2hj} becomes $v^\e(y,0)=f (F_\e(y))$, $y\in M_\e$, when we
interpret $``\e y$''$ := F_\e(y)$.

\subsection{Viscosity solutions.}\quad

Hamilton-Jacobi equations usually do not have smooth solutions due to 
the intersection or ``shocks'' of their characteristic lines. Local (smooth) solutions
can be found using the characteristic method. 
Patching local solutions can give a plethora of almost everywhere solutions.
A way to recover existence and uniqueness results for the Hamilton-Jacobi equation 
\begin{equation}\label{HJ}
\partial_t u + H(x, \partial _x u)=0, 
\qquad u:\Om\to\re,
\qquad \Om=M\times[0,\infty[.
\end{equation}
is to consider viscosity solutions.

Write $\Om:=M\times [0,\infty[$.
A continuous function $u\in C^0(\Om,\re)$ is a {\it viscosity subsolution} of the 
Hamilton-Jacobi equation \eqref{HJ} if
\begin{gather}
\forall \phi \in C^1(\Om,\re),
\; \text{\it if $(x_0,t_0)\in \Om$ is a local maximum of $u-\phi$, }
\notag
\\
then \quad \partial_t\phi(x_0,t_0) + H\big(x_0,\partial_x \phi(x_0,t_0)\big)\le 0.
\label{subsol}
\end{gather}

A continuous function $u\in C^0(\Om,\re)$ is a {\it viscosity supersolution} of the 
Hamilton-Jacobi equation \eqref{HJ} if
\begin{equation}\label{supersol}
\begin{gathered}
\forall \phi \in C^1(\Om,\re),
\; \text{\it if $(x_0,t_0)\in \Om$ is a local minimum of $u-\phi$, }
\\
then \quad \partial_t \phi(x_0,t_0) + H\big(x_0,\partial_x \phi(x_0,t_0)\big)\ge 0.
\end{gathered}
\end{equation}

A function $u\in C^0(\Om,\re)$ is a {\it viscosity solution} of the
Hamilton-Jacobi equation \eqref{HJ} if it is both a viscosity subsolution and
a viscosity supersolution of \eqref{HJ}.

\begin{Lemma}\label{l03}\quad

In the definition of viscosity subsolution  (supersolution)
it is equivalent
\begin{enumerate}
\item
 to take a {\sl strict} local maximum ({\sl strict} local minimum).
 \label{l11}
\item 
to take Lispchitz $C^1$ test functions.
\label{l12}
\end{enumerate}
\end{Lemma}
\begin{proof}\quad

\eqref{l11}. Suppose  the definition for subsolution holds for strict local maxima.
If $u-\vr$ with $\phi\in C^1$ has a local maximum at $(x,t)$, define in local coordinates 
$\phi(y,s):=\vr(y,t)+|x-y|^2+|s-t|^2$. Then $u-\phi$ has a strict local maximum at $(x,t)$ and
 inequality~\eqref{subsol} holds for $\phi$. But the derivatives of $\phi$ and $\vr$ at $(x,t)$ are equal.
 Then inequality~\eqref{subsol} holds for $\vr$ at $(x,t)$.
 The converse is easier.
 
 \eqref{l12}. Suppose the definition for subsolutions holds for Lipschitz test functions.
 Given a test function $\phi$ since  $\phi\in C^1(\Om,\re)$ we have that $\phi$ 
 is Lipschitz in a neighborhood of $U$ of $(x_0,t_0)$. We can extend $\phi|_U$ 
 to a Lipschitz $C^1$ function $\vr$ on $\Om$. Then inequality \eqref{subsol} holds
 for $\vr$ and $D\vr(x_0,t_0)=D\phi(x_0,t_0)$. Thus inequality \eqref{subsol} holds for $\phi$.

\end{proof}

We shall prove the following Theorem  using Evans
perturbed test function method.
 
\begin{Theorem}[Contreras, Iturriaga, Siconolfi \cite{CIS}]\label{T1}\quad

Let $M$ be a closed Riemannian manifold.
Let $H:T^*M\to\re$ be a Tonelli Hamiltonian and
$f_\e:(M_\e,d_\e)\to \re$ an equi-Lipschitz family such that
$\lim_\e f_\e= f$ uniformly, with 
\linebreak 
$f:H_1(M,\re)\to\re$. 

Let $\hH$ be the lift of $H$ to $\hM$ and let $v^\e$ be the
Lipschitz viscosity solution to the problem
\begin{equation}\label{eve}
\begin{gathered}
\partial_t v^\e+\hH\big(y,\tfrac 1\e \partial_y v^\e\big) =0, \\
v^\e(y,0)=f_\e(y).
\end{gathered}
\end{equation}
Then the family $v^\e:\hM_\e\times]0,+\infty[\to\re$ is 
equi-Lipschitz and
 $$
 \lim_{\e\to 0}v^\e=u:H_1(M,\re)\to\re
 $$
uniformly on compact sets of $H_1(M,\re)\times]0,+\infty[$,
where $u$ is the solution to
\begin{equation}\label{hhj}
\begin{gathered}
\partial_t u +\oH(\partial_x u) = 0,\\
u(x,0)=f(x);
\end{gathered}
\end{equation}
and $\oH:H^1(M,\re)\to\re$ is $\oH=\a$ Mather's alpha function.
\end{Theorem}

    Recall from \cite{CIPP} Corollary~1, that Mather's alpha function is
    $\a:H^1(M,\re)\to\re$
    $$
    \a(c) = \inf_{[\om]=c} \sup_{x\in M} H(x,\om(x)),
    $$
    where $\om$ are closed 1-forms of cohomology $c$. Mather's alpha function is also the
    convex dual of Mather's beta function, and it can be written in terms of minimizing invariant measures,
    see Mather~\cite{Mat5}, Theorem~1. 
    
    The same proof applies to (abelian) subcovers of $\hM$
    with the statement given in \cite{CIS}, Theorem~1.4.
    These may be more intuitive than $\hM$. For example,
    let $M=\T^2\#\T^2$ be the bi-torus, the compact oriented surface of genus 2.
    Cut $M$ along the ``ears'' to get a surface $F=(\SS^1\times I)\#(\SS^1\times I)$
    homeomorphic to the boundary of a tubular neighborhood of a cross in $\re^3$,
    which is a fundamental domain of an abelian cover $\M$ of $M$ with group of deck
    transformations $\Z^2$.
    The cover $\M$ is homeomorphic to the boundary of a small tubular neighborhood of
    $$
    (\re\times \Z\times\{0\}) \cup (\Z\times\re\times\{0\})
    $$ 
    in $\re^3$. Nevertheless the maximal abelian cover $\hM$
    has group of deck transformations $\Z^4=H_1(M,\Z)$.
    In this example $\M_\e=(\M,\e\, d_\M)$ converges to $\re^2$.
    We see that the limit process of $\M_\e$ destroys all the differentiable
    structure of $\M$. Nevertheless there is the homogenization limit because
    the Hamilton-Jacobi equation responds to a large scale variational principle.

    \subsection{The cell problem.}\quad

    The cell problem in the torus $\T^n$ is written as
    \begin{equation}\label{cellT}
    H(x,P+Dw(x)) = \oH(P), \qquad 
    x\in\T^n,\quad P\in\re^n, \quad w:\T^n\to\re.
    \end{equation}
    The real number $\oH(P)=\a(P)$ is the unique constant for which
    equation \eqref{cellT} has a viscosity solution. Thus the cell problem consists
    in finding a ``correcting'' exact form $d_xw$ such that the form $P+d_xw$
    is a solution of the Hamilton-Jacobi equation in the cohomology class $P$.
    
    In the manifold case we don't have a standard canonical basis for $H^1(M,\re)$.
    So we use the same basis $c_i=[\om_i]$ used for the isomorphism 
    $H_1(M,\re)=\H^1(M,\re)^*$ and the definition of $G$ in \eqref{dfg}.
    Let $\Om^1(M)$ be the space of smooth closed 1-forms in $M$.
    Let $g:H^1(M,\re)\to \Om^1(M)$ be 
    $$
    g\Big (\tsum_i \,p_i c_i\Big) =\tsum_i\, p_i \om_i.
    $$
    Writing $c_i\in H^1(M,\re)=H_1(M,\re)^*$, the definition~\eqref{dfg}
    of $G:\hM\to H_1(M,\re)$ is equivalent to 
    $$
    c_i \cdot G(x) =\int_{x_0}^x \whom_i.
    $$
    Then the i-th coordinate of the derivative $DG$ applied on the vector 
    $h\in T_x\hM$ is
    $$
    (c_i\cdot DG(x))(h) = \whom_i(x)\cdot h.
    $$
    This implies that
    \begin{equation}\label{NG}
    \sup_{x\in\hM}\lV DG(x)\rV\le \max_i \;\max_{x\in M}\lV \om_i(x)\rV.
    \end{equation}
     For a cohomology class $P=\sum_i p_i c_i\in H^1(M,\re)$ we have that
    \begin{align}
    P\cdot DG(x) &=\sum p_i \;c_i\cdot DG(x)
    =\sum p_i\,\whom_i(x)
    \notag\\
    &=\sum p_i \, \om_i(\pi(x))\circ d\pi 
    = g(P)(\pi x)\circ d\pi
    =\pi^* (g(P))(x),
    \label{pdg}
    \end{align}
    where $\pi:\hM\to M$ is the covering map.
    
    Given $P\in H^1(M,\re)$ our cell problem will be
    \begin{equation}\label{cell}
    H\big (x, g(P)+D w(x)\big) =\oH(P),
    \quad w:M\to\re.
    \end{equation}
    Using \eqref{pdg}, a solution of the cell problem \eqref{cell}
    lifts to $\hM$ as $\hw=w\circ \pi$
    which satisfies
    \begin{equation}\label{lcell}
    \hH(x,P\cdot DG(x)+ D\hw(x)) = \oH(P),
    \qquad x\in \hM.
    \end{equation}
    because $P\cdot DG = \pi^* (g(P))$
    and $D\hw = \pi^*(Dw)$.

\section{The projections $F_\e$.}

The universal cover $\tM$ of $M$ is the set of homotopy classes with fixed endpoints
of the curves in $C^0(([0,1],0),(M,\pi x_0))$. The projection
$p:\tM\to M$ is $p([\ga])=\ga(1)$.
We also name $x_0\in \tM$ the class of the constant curve $\ga(t)\equiv \pi x_0$ in $M$. 

The maximal abelian cover or universal abelian cover
$\hM$ is the subcover given by $\tM/_\sim$, where $\sim$ is the equivalence relation given by
$[\ga]\sim[\eta]$ iff $\ga(1)=\eta(1)$ and $\ga*\eta^{-1}$ is null homologous (in the base $M$),
 or equivalently, 
its homotopy class is in the center $H=[\pi_1,\pi_1]\subset \pi_1 (M,\pi x_0)$ of the fundamental group.
This implies that a continuous curve $\Ga:[0,1]\to \hM$, $\Ga(0)=x_0$
\begin{equation}\label{PH}
\text{ is closed iff   its projection
to the base $M$ has homology zero. }
\end{equation}
Statement \eqref{PH}  holds at  other base points because the subgroup 
$H\lhd \pi_1(M,\pi x_0)$ is normal.
Observe that $H_1(\hM,\Z)$ may be non-trivial, i.e. a closed curve in $\hM$ 
may not be homologous to zero in $\hM$, but its projection to $M$ must be 
homologous to zero in $M$. The projection $\pi:\hM\to M$ is also $\pi([\ga])=\ga(1)$.
The metric on $\hM$ is obtained by lifting  the riemannian metric on $M$ using $\pi$.

Property \eqref{PH} also implies that if a (closed) form $\widehat\om\in \Om^1(\hM)$ is a lift 
of a closed 1-form $\om\in\Om^1(M)$ in the base $M$, i.e. $\widehat\om=\pi^*(\om)$,
then the function $f_\om:\hM\to\re$,
$$
f(x):=\oint_{x_0}^x \widehat\om,
$$
is well defined in $\hM$, i.e. it does not depend on the curve $\hga:x_0\to x$
used to compute the integral, because for another curve $\heta:x_0\to x$
$$
\oint_{\hga-\heta}\widehat \om  = \int_{\pi\circ\hga\,-\,\pi\circ\heta} \om
=0,
$$
because the chain $\pi\circ\hga-\pi\circ\heta$ is homologous to zero.

\bigskip
\begin{Lemma}\label{lLip}
\quad

%\begin{minipage}{10cm}
\hskip .5cm
\parbox{10cm}{
\begin{enumerate}[(a)]
\item The map $G:\hM\to H_1(M,\re)$  in \eqref{dfg}, is Lipschitz.
\item The map $F_\e:(M_\e,d_\e)\to H_1(M,\re)$ in \eqref{Fe}, is Lipschitz.
\end{enumerate}
}
%\end{minipage}
\end{Lemma}
\begin{proof}\quad

(a). For the $i$-th coordinate of $G(x)$ we have that
$$
\big|G(y)\cdot c_i -G(x)\cdot c_i\big|
=\left|\oint_{x_0}^y \widehat\om_i - \oint_{x_0}^x \widehat\om_i\right|
\le \lV \widehat\om_i\rV d(x,y)
\le \lV \om_i\rV_M d(x,y).
$$
Thus $G$ has Lipschitz constant $K:=\max_i \lV \om_i\rV$.

(b).
\begin{align*}
\lv F_\e(x)-F_\e(y)\rv
&= \lv \e\, G(x) -\e\, G(y)\rv
\le \e \, K\, d_{\hM}(x,y)
\\
&\le K\, d_\e(x,y).
\end{align*}

\end{proof}

\begin{Proposition}
The map $F_\e:M_\e\to H_1(M,\re)$ is proper.
\end{Proposition}

\begin{proof}\quad

We have to prove that the pre-image of compact sets are compact.
Since $F_\e$ is continuous it is enough to prove that the pre-image of
bounded sets are bounded.
Since $F_\e=\e\, G$ and $d_\e = \e \,d_\hM$, it is enough to prove it for $G$.

By the classification of finitely generated abelian groups there is an isomorphism
\linebreak
$f:H_1(M,\Z)\to \Z^k\oplus \T$, where the torsion
$\T=\Z_{a_1}\oplus\cdots\oplus \Z_{a_q}$ is the finite group
of the elements in $H_1(M,\Z)$ with finite order:
$$
\T=\{\,h\in H_1(M,\Z)\;:\; \exists m\in\na\quad m\cdot h=0\;\}.
$$

Let $\ga_i$ be a closed curve in $M$ based at $\pi x_0\in M$
representing the homology class 
\linebreak
$[\ga_i]=f^{-1}(e_i\oplus 0)$, 
where $e_i=(0,\cdots,1,\cdots, 0)$ are the canonical basis vectors  of $\Z^k$.

On $H_1(M,\re)\approx\re^k$ we are using the coordinates $c_i\cdot h$ and we use the norm
\begin{equation}\label{hnorm}
\lV h\rV = \tsum_{i=1}^k |c_i\cdot h|.
\end{equation}
We have that a  basis for $H_1(M,\re)$ is $([\ga_i] )_{i=1}^k$.
If $h = \sum_{i=1}^k r_j [\ga_j]$ then 
$$
c_i\cdot h = \tsum_{j=1}^k r_j \, (c_i\cdot[\ga_j])
= \tsum_{j=1}^k a_{ij} r_j,
$$
where 
$$
a_{ij} = c_i\cdot[\ga_j]=  \oint_{\ga_j}\om_i.
$$
Since this is a change of basis the matrix $A=[a_{ij}]$ is non-singular.
The norm \eqref{hnorm} of $h$ is
$
\lV h\rV = |A \ov r|.
$

The group of deck transformations of $\hM$
is isomorphic to $H_1(M,\Z)$. We write $x\mapsto \zeta\cdot x$
 the deck transformation associated to $\zeta\in H_1(M,\Z)$.
%Define
%$$
%Q_1 :=\diam M + \max\{ \,d(x_0,\tau\cdot x_0) : \tau\in \T\,\}.
%$$

Suppose that $x\in\hM$ satisfies 
\begin{equation*}
\lV G(x)\rV\le R.
\end{equation*}

Let $\a:[0,1]\to M$ be a minimizing geodesic in $M$ joining $\pi(x)$ to $\pi x_0$.
Let $\ha$ be the lift of $\a$ to $\hM$ with $\ha(0)=x$.
We have that $\ha(1)\in \pi^{-1}(\pi x_0)=H_1(M,\Z)\cdot x_0$ and
$d(x,\ha(1)) \le \ell(\a)= d(\pi(x),\pi x_0)\le \diam M$.
Therefore there are $\on\in\Z^k$ and $\tau\in \T$ such that 
\begin{equation}\label{16n}
d_{\hM}(x,(\on+\tau)\cdot x_0)\le \diam M. 
\end{equation}

Since the torsion group $\T$ is finite we have that 
\begin{equation}\label{Q1}
Q_1 := \diam M + \max\{\, d(x_0,\tau\cdot x_0) : \tau \in \T\,\}<\infty.
\end{equation}
Since the deck transformations are isometries we have that 
\begin{align*}
d_\hM(x, \on\cdot x_0)  
&\le \diam M + (Q_1-\diam M) =Q_1.
\end{align*}
\begin{align}
c_i\cdot G(\on\cdot x_0) &=c_i\cdot G(x) + \oint_x^{\on\cdot x_0} w_i
\notag\\
&\le c_i\cdot G(x) + Q_1 \lV \om_i\rV,
\notag\\
\lV G(\on\cdot x_0)\rV &\le \lV G(x)\rV+ k \,Q_1 \, \max_{1\le i\le k}\lV \om_i\rV
\notag\\
&\le R + B,
\label{nGn}
\end{align}
where $B:= k \,Q_1 \, \max_{1\le i\le k}\lV \om_i\rV$.

Let $\La:[0,1]\to \hM$ be a minimizing geodesic in $\hM$ from $x_0$ to 
$\on\cdot x_0$.
Its projection $\la=\pi\circ\La$ has homology class
$\on  = \sum_i n_i [\ga_i]\in\Z^k\oplus 0 \subset  H_1(M,\Z)$. 
Let $\Ga: = \ga_1^{n_1}*\cdots*\ga_k^{n_k}$,
a loop in $(M,\pi x_0)$ with homology $\on$.
Let $\hGa^{-1}$ be a lift of $\Ga^{-1}$
with $\hGa^{-1}(0)=\on\cdot x_0$.
Since $\La*\hGa^{-1}$ has projection $\la*\Ga^{-1}$ 
homologous to zero, the curve $\La*\Ga^{-1}$ is a closed
loop  in $(\hM,x_0)$.
Then 
\begin{align}
d_{\hM}(x_0,\on\cdot x_0)&=d(x_0,\La(1))=d(\hGa^{-1}(1),\hGa^{-1}(0))
\notag\\
&\le \length(\Ga) \le \tsum_{i=1}^k |n_i|\; \ell(\ga_i).
\label{20}
\end{align}

The coordinates of $G(\on\cdot x_0)$ are 
$$
c_i\cdot G(\on\cdot x_0)
= c_i\cdot \tsum_{j=1}^k n_j [\ga_j]
=\tsum_{j=1}^k a_{ij} \, n_j,
\qquad 
\on = A^{-1} [c_i\cdot G(\on\cdot x_0)]
$$
$$
|\on| \le \lV A^{-1} \rV \lV G(\on\cdot x_0)\rV
\le \lV A^{-1}\rV (R+B),
\qquad\text{using \eqref{nGn}.}
$$
\begin{align*}
d_\hM(x_0,x) 
&\le d(x_0,\on\cdot x_0) + d(\on\cdot x_0, x)
\\
&\le |\on| \max_i \ell(\ga_i) + Q_1.
\\
d_\hM(x_0,x)&\le \lV A^{-1} \rV (R+B) \max_i \ell(\ga_i) +Q_1<\infty.
\end{align*}
Therefore $G^{-1} (\{ h\in H_1(M,\re) :\lV h\rV\le R\})$ is bounded.

\end{proof}

\begin{Corollary}\label{c1.3}\quad

If $\e>0$, $(y_0,t_0)\in M_\e\times\re^+$, $0<r<t_0$
then 
$$
(F_\e\times id)^{-1}\big(\,\ov{B_r(y_0,t_0)}\,\big)\subset M_\e\times \re^+
\quad \text{
is compact.}
$$
\end{Corollary}

\bigskip

\section{Lipschitz}

Let $L:TM\to \re$ be the lagrangian of $H$
$$
L(x,v):=\sup\{\,p(v)-H(x,p)\;:\;p\in T_x^*M\,\} \qquad (x,v)\in TM;
$$
and let $L^\e:T\hM\to\re$ be the lagrangian of the hamiltonian $H^\e(x,p)=\hH(x,\tfrac 1\e p)$, 
$$
L^\e(x,v) := \sup \big\{\,p(v) -H^\e(x,p)\;:\;p\in T^*\hM\,\big\},\qquad (x,v)\in T\hM.
$$
Then $L$ and $L^\e$ are also convex and superlinear.
Observe that for $\hL=L\circ d\pi$,
\begin{equation}\label{Le}
L^\e(x,v) = \hL(x,\e v).
\end{equation}
Then the energy function 
$E^\e(x,v) = v \cdot \partial_v L^\e(x,v) -L^\e(x,v)$
is 
\begin{equation}\label{Ee}
E^\e (x,v) = \hE(x,\e v),
\end{equation}
where $E=v\cdot L_v -L$ is the energy function of $L$
and $\hE=E\circ d\pi$.

The solution to equation~\eqref{eve} is given by the Lax formula
\begin{equation}\label{uef}
u^\e(x,t) = \min\left\{ f_\e(y) +\int_0^t L^\e(\ga,\dga) : \ga\in C^1(([0,t],0,t);(M_\e,y,x))\,\right\}.
\end{equation}
In section \S\ref{Suniqueness}, Corollary~\ref{Cuniqueness}, we prove that there is a unique 
Lipshitz viscosity solution to the problem~\eqref{eve} and in Proposition~\ref{eqlip} 
we prove that $u^\e$ is Lipschitz. So we will have that $u^\e=v^\e$ 
is the solution to problem \eqref{eve} in Theorem~\ref{T1}.

Define the Lax-Oleinik operator $\cL^\e_t$ by
\begin{equation}\label{Lax}
(\cL^\e_t f)(x) = \min\left\{ f(y) +\int_0^t L^\e(\ga,\dga) : \ga\in C^1(([0,t],0,t);(M_\e,y,x))\,\right\}.
\end{equation}
The minimum in \eqref{Lax} is always attained by one or more minimizers $\ga$.
Then $\cL_t$ is a semigroup, meaning that $\cL^\e_t\circ\cL^\e_s = \cL^\e_{s+t}$.
 We have that
$$
u^\e(x,t) = (\cL^\e_t f_\e)(x),
$$
and then 
\begin{equation}\label{17}
\forall s \le t \quad \forall x\in M_\e\qquad u^\e(x,t) = \cL^\e_{t-s} u^\e(\cdot,s)(x).
\end{equation}

\bigskip

\begin{Proposition}\label{eqlip}\quad

The functions $u^\e$ in \eqref{uef} are equi-Lipschitz,
i.e. 
\begin{align*}
\exists Q>0 \quad \forall \e>0 \quad \forall (x&,t),(y,s)\in M_\e\times[0,+\infty[
\\
&\lv u^\e(x,t)-u^\e(y,s)\rv
\le Q\big[|t-s|+d_\e(x,y)\big].
\end{align*}
\end{Proposition}

\bigskip
\begin{proof}\quad

{\bf Step 1:}
{\it There are 
\begin{equation}\label{k0a0}
k_0>0, \quad a_0>0
\end{equation}
such that all the minimizers $\ga$ in \eqref{uef}}
satisfy $E^\e(\ga,\dga)<k_0$ and $\e\,|\dga|<a_0$.
\medskip

Taking a constant curve $\ga(s)= x$, $s\in[0,t]$ in \eqref{uef},
we have that 
\begin{equation}\label{ue<}
\forall (x,t) \in M_\e\times\re^+\qquad 
u^\e(x,t) \le f_\e(x) + \hL(x,0) \,t.
\end{equation}

Let $K$ be a uniform Lipschitz constant for the functions $f_\e$:
$$
\forall x,y\in M_\e\qquad 
|f_\e(x)-f_\e(y)|\le K d_\e(x,y) = K\, \e\, d(x,y)
$$
Let $A>0$ be such that 
\begin{equation}\label{A1}
A > K.
\end{equation}
By the superlinearity of $L$ there is $B>0$ such that 
\begin{equation*}
\forall (x,v) \in TM \qquad L(x,v) > 2A\, |v|-B.
\end{equation*}
Therefore
\begin{equation*}
\forall (x,v) \in TM_\e \qquad L^\e(x,v) > 2A\,\e\, |v|-B.
\end{equation*}
Let $a_1>0$ be such that 
\begin{equation}\label{a1}
A\, a_1 - B > \sup_{x\in M} |L(x,0)|.
\end{equation}
Let $k_0>0$ be such that 
$$
(x,v)\in TM, \quad 
E(x,v) \ge k_0 \quad \then \quad |v|> a_1.
$$
Then 
\begin{equation}\label{ek0}
\hskip 0.8cm
E^\e (x, v)=\hE(x,\e v) \ge k_0 \quad \then \quad
\e\, |v| > a_1.
\end{equation}
Let $a_0>0$ be such that 
\begin{equation}\label{fk0}
E(x,v)\le k_0 \quad \then \quad |v|<a_0.
\end{equation}

Given  $y\in M_\e$ if a curve  $\ga\in C^1(([0,t],0,t);(M_\e,y,x))$
has energy $E^\e(\ga,\dga)\ge k_0$,
by \eqref{ek0} it has speed
$$
\forall s\in[0,t] \qquad \e \,|\dga(s)| > a_1.
$$  
For $t>0$ we have that 
\begin{align*}
\int_0^tL^\e(\ga,\dga)&\ge \int_0^t 2 A\, \e |\dga| -B
& &&
\\
&\ge \int _0^t A\, \e |\dga| + \int_0^t A\,\e\left( \frac{\,a_1}\e\right) - B
&&
\\
&\ge A \, \e d(x,y) + (A\, a_1 - B) t
&&
\\
&> A \, d_\e(x,y) + \hL(x,0) \,t
&& {\text{by \eqref{a1}. }}
\end{align*}

\begin{align*}
f_\e(y) +\int_0^t L^\e(\ga,\dga)
&> 
f_\e(x) -K \,d_\e(x,y) +A\, d_\e(x,y) + \hL(x,0)\,t
\\
&=
f_\e(x) +(A-K)\, d_\e(x,y) + \hL(x,0)\,t
\\
&\ge f_\e(x) + \hL(x,0)\, t
\\
&\ge u^\e(x,t) \hskip 3.3cm \text{by \eqref{ue<}.}
\end{align*}
Therefore $\dga$ can not be a minimizer in \eqref{uef}.
Thus any minimizer $\ga$ must 
have energy $E^\e(\ga,\dga)< k_0$.
By \eqref{fk0} it must 
satisfy $\forall s\in[0,t]$ \;$\e\,|\dga(s)|< a_0$.

\bigskip
\bigskip

{\bf Step 2:} {\it The functions $u^\e(x,t)$ are equi-Lipschitz in $x$.}
\medskip

We have to prove that there is a uniform Lipschitz constant 
for the functions 
$$
(M_\e,d_\e) \ni x \longmapsto u^\e(x,t),
\qquad  \e>0, \quad t\ge 0.
$$ 
By hypothesis for $t=0$, the functions $u^\e(x,0):=f_\e(x)$ have uniform Lipschitz constant $K$ in $(M_\e,d_\e)$.
Observe that it is enough to prove that for all $(x,t)\in M_\e\times \re^+$
there is a uniform {\it local} Lipschitz constant for $(M_\e,d_\e)\ni x\mapsto u^\e(x,t)$. 

Since the projection $M$ is compact, by Weiestrass Theorem 
(cf. Mather \cite[p. 175]{Mat5}) 
\begin{equation}\label{weier1}
\forall A>0 \qquad \exists \tau=\tau(\e,A)>0 
\end{equation}
such that if $s_0<t\le s_0+\tau$ and $x,y\in M_\e$ satisfy $d_\e(x,y)< \tfrac 12 A(t-s_0)$ 
then there is a unique minimizer $\zeta$ of the $L^\e$-action in
$C^1(([s_0,t],s_0,t);(M_\e,x,y))$
and moreover
$|\dga|_\e\le A$.

Let $a_0$ be from \eqref{k0a0} in  {\sc step 1} and take $\tau=\tau(\e,4a_0)$,
 from \eqref{weier1}. 
Let $t>0$. Shrinking $\tau$ if necessary we can assume that $2\tau< t$.
Let $s_0=t-\tau$.
Let $x,y\in M_\e$ be such that 
$$
d_\e(x,y)< a_0\,\tau.
$$
Let  $\ga:[0,t]\to M_\e$ be the minimizer in \eqref{uef} such that 
$$
u^\e(x,t) = f_\e(\ga(0))+\int_0^t L^\e(\ga,\dga).
$$
By  {\sc step 1}
$|\dga|<a_0$ and then $d_\e(\ga(s_0),\ga(t))<a_0\,\tau$.
Let $\eta\in C^1(([0,1],0,1); (M,x,y))$ a minimizing geodesic from $x$ to $y$.
In particular
\begin{equation}\label{deta}
\forall s\in[0,1]\qquad |\partial_s\eta(s)|= d(x,y). 
\end{equation}
Let $h:[0,1]\times[s_0,t]\to M$ be the variation of $\ga|_{[s_0,t]}$ such that 
for every $s\in[0,1]$, $h(s,\cdot)$ is the minimizer joining
$\ga(s_0)$ to $\eta(s)$.
Since $\ga(t)=x$ we have that 
\begin{equation*}\label{2a0t}
d(\ga(s_0),\eta(s))\le d(\ga(s_0),\ga(t))+d(x,y)<2a_0\,\tau.
\end{equation*}
Then from~\eqref{weier1}
\begin{equation}\label{4a0}
\forall(\si,\tau)\in[0,1]\times[s_0,t]\qquad
\e\,|\partial_\tau h(\si,\tau)|\le 4 a_0.
\end{equation}
Let 
\begin{equation}\label{30b}
b_2 > K \ge  \sup_\e \Lip(f_\e,d_\e).
\end{equation}
 be such that 
\begin{equation}
(x,w)\in TM \qquad
|w|\le 4a_0
\quad \then \quad
| L_v (x,w)| < b_2.
\label{b2}
\end{equation}
Observe that $b_2$ is independent of $(\e,t)$. From~\eqref{4a0} and~\eqref{b2} we get that
\begin{align}\label{lvb2}
\forall (\si,\tau)\in[0,1]\times[s_0,t]\qquad
\big| \hL_v\big(h(\si,\tau),\e\,\partial_\tau h(\si,\tau)\big)\big| < b_2.
\end{align}
Observe that 
\begin{align*}
u^\e(h(\si,t),t) 
&\le u^\e(\ga(s_0),s_0) +\int_{s_0}^t
L^\e\big(h(\si,\tau),\partial_\tau h(\si,\tau)\big)\;d\tau,
\\
u^\e(x,t) 
&= u^\e(\ga(s_0),s_0)
+\int_{s_0}^t
L^\e\big( h(0,\tau),\partial_\tau h(0,\tau)\big) \;d\tau.
\end{align*}

\begin{align}
u^\e(\eta(\si),t)-u^\e(x,t)
&\le \int_{s_0}^t \big[ L^\e\big( h(\si,\tau),\partial_\tau h(\si,\tau)\big)
-L^\e\big( h(0,\tau),\partial_\tau h(0,\tau)\big)\big]
\; d\tau,
\notag \\
u^\e(y,t)-u^\e(x,t) &\le
\int_0^1 
\int_{s_0}^t (L^\e_x \, \partial_\si h + L^\e_v \, \partial_\si\partial_\tau h) \; d \tau \; d\si .
\label{duu}
\end{align}
Since $\partial_\si\partial_\tau h=\partial_\tau\partial_\si h$, integrating by parts we get
\begin{align*}
\int_{s_0}^t (L^\e_x \, \partial_\si h + L^\e_v \, \partial_\si\partial_\tau h) \; d \tau \; d\si 
&=
\int_{s_0}^t (L^\e_x   - \tfrac d{d\tau}L^\e_v )\,  \, \partial_\si h 
+ \tfrac d{d\tau}( L^\e_v 
\partial_\si h) \; d \tau .
\end{align*}
Since $\tau\mapsto h(\si,\tau)$ is a solution of the Euler-Lagrange equation, the first term is zero.
The integral of the second term is 
\begin{align*}
\int_{s_0}^t (L^\e_x \, \partial_\si h + L^\e_v \, \partial_\si\partial_\tau h) \; d \tau \; d\si 
&=
L^\e_v\,\partial_\si h \Big\vert_{s_0}^t
= L^\e_v\big(h(\si,t), \partial_\tau h(\si,t)\big)\; \partial_\si\eta(\si)
\end{align*}
because $\partial_\si h(\si,s_0)\equiv 0$ and $\partial_\si h(\si,t)=\partial_\si \eta(\si)$.
Then from \eqref{duu}  we have that
\begin{align*}\label{duu1}
u^\e(y,t)-u^\e(x,t) 
&\le
\int_0^1  |L^\e_v\big(h(\si,t), \partial_\tau h(\si,t)\big)|\; |\partial_\si\eta(\si)|
\; d\si.
\end{align*}
Observe that 
$\partial_vL^\e(x,v)=\e\, \hL_v(x,\e v)$,
 then using \eqref{lvb2} and \eqref{deta} 
we get that 
$$
d_\e(x,y)<a_0\tau \quad \then \quad 
u^\e(y,t)-u^\e(x,t) \le \e\, b_2 \; d(x,y)= b_2\; d_\e(x,y).
$$
Since $a_0\tau$ does not depend on $(x,y)$,
we can interchange the roles of $x$ and $y$ and obtain
$$
d_\e(x,y)<a_0\;\tau(\e,a_0,t)
\quad \then \quad
|u^\e(x,t)-u^\e(y,t)|\le b_2\;d_\e(x,y).
$$
This implies that
\begin{equation}\label{ss2}
\forall x,y\in M_\e
\qquad
|u^\e(x,t)-u^\e(y,t)|\le  b_2\;d_\e(x,y).
\end{equation}
Observe that from  \eqref{b2}
the constant 
$b_2$ does not depend on $t$.

\bigskip
{\bf Step 3:} {\it The functions $u^\e(x,t)$ are equi-Lipschitz in $t$.}
\medskip

Comparing with a constant curve in $x$ the definition of $u^\e$ in~\eqref{uef}  implies that 
\begin{equation}\label{34}
x\in M_\e,\quad 
t\ge s 
\quad \then\quad
u^\e(x,t) \le u^\e(x,s) + \hL(x,0)\,(t-s),
\end{equation}
because $L^\e(x,0)=\hL(x,0)$.
Let $c_2>0$ be such that 
$$
\forall (x,v)\in TM \qquad
L(x,v) > b_2 \,|v| -c_2.
$$
Then
\begin{equation}\label{ss3}
\forall (x,v)\in TM_\e \qquad
L^\e(x,v) > b_2\,\e \,|v| -c_2.
\end{equation}
Let $\ga\in C^1([s,t], M)$ be a curve with $\ga(t)=x$ where the minimum in 
\eqref{17} is attained, and let $y:=\ga(s)$. Using \eqref{ss2} and \eqref{ss3},
\begin{align}
u^\e(x,t) &= u^\e(\ga(s),s) +\int_s^t L^\e(\ga,\dga)
\notag\\
&\ge u^\e(x,s) - b_2\, d_\e(x,y) + \int_s^t b_2\, \e\, |\dga| \; d\tau - c_2 \,(t-s)
\notag\\
&\ge u^\e(x,s) -b_2\, d_\e(x,y)+b_2\, \e\, d(x,y) - c_2(t-s)
\notag\\
&\ge u^\e(x,s) - c_2 \,(t-s).
\label{c2}
\end{align}
Taking $c_1:=\max\{c_2,\max_{x\in M}|L(x,0)|\}$
from \eqref{34} and \eqref{c2},
$t\mapsto u^\e(x,t)$ has the same  Lipschitz
constant $c_1$ for every $(\e,x)$.

\end{proof}

\section{Equicontinuity.}

We need the Arzel\`a-Ascoli Theorem in our context.

Let $(M_n,d_n)_{n\in\na}$ be a sequence of complete metric spaces. Suppose that 
$\lim_n(M_n,d_n)=(M,d)$ and that $(M,d)$ is separable.
We say that a family of functions $f_n:(M_n,d_n)\to \re$
is {\it bounded on compact sets} if for every compact set $K\subset M$
the set $\cup_{n\in\na} f_n(F_n^{-1}(K))$ is bounded.
Recall that the uniform convergence on compact sets is defined in~\eqref{ucs}.

We say that the family $f_n$ is {\it equicontinuous} if
$$
\forall \e>0 \quad \exists \de>0
\quad\forall n\in \na
\quad d_n(x,y)<\de \quad\then\quad
|f_n(x)-f_n(y)|<\e.
$$

\begin{Prop}\quad\label{Paa}

If $\lim_n (M_n,d_n,F_n)=(M,d)$, with $(M,d)$  separable 
 and $f_n:(M_n,d_n)\to\re$ is equicontinuous 
and bounded on compact sets,
then there is a subsequence $n(k)$ such that 
$f_{n(k)}$ converges uniformly on compact sets
to a continuous function on $M$. 
\end{Prop}

\begin{proof}
Let $\D=\{x_n\}_{n\in\na}$ be a countable dense subset of $M$.
By~\eqref{1b} for all $n\in\na$ there are $x^m_n\in M_m$ such that $\lim_m F_m(x^m_n)=x_n$.
Then for each $n$, the set $\cup_m \{F_m(x^m_n)\}\cup\{x_n\}$ is compact and hence the
sequence $\{ f_m(x^m_n)\}_m$ is bounded. 
Let $m(1,k)\in\na$ be an increasing sequence such that 
$$
\exists \lim_k f_{m(1,k)}(x^{m(1,k)}_1) =: a_1.
$$
Let $m(2,k)$ be an increasing subsequence of $m(1,k)$ such that 
$$
\exists \lim_k f_{m(2,k)}(x^{m(2,k)}_2) =: a_2.
$$
Inductively, let $m(\ell,k)$ be an increasing subsequence of $m(\ell-1,k)$ such that 
\begin{equation}\label{fmL}
\exists \lim_k f_{m(\ell,k)}(x^{m(\ell,k)}_\ell) =: a_\ell.
\end{equation}
Then $m(k,k)\in\{ m(\ell,k)\}_{k\ge \ell}$ and hence $\{m(k,k)\}_{k\in\na}$ is a subsequence of 
all the sequences $m(\ell,\cdot)$. In particular 
\begin{equation}\label{akkn}
\forall n\in\na \qquad\lim_k f_{m(k,k)}(x^{m(k,k)}_n) = a_n.
\end{equation}

Define $f:\D\to \re$ by $f(x_n):=a_n$. We first prove that $f|_\D$ 
is uniformly continuous. Given $\e>0$ let $\de>0$ be such that
\begin{equation}\label{equic}
\forall n\in\na\qquad
d_n(x,y)<\de \quad \then\quad |f_n(x)-f_n(y)|< \e.
\end{equation}

Suppose that $x_p,\,x_q\in \D$ and $d(x_p,x_q)<\de$.
From~\eqref{1a} we have that 
\begin{align}
B^{-1}\, d_m(x^m_p,x^m_q)  - A_m
&\le d(F_m(x^m_p),F_m(x^m_q))
\le B\, d(x^m_p,x^m_q).
\notag\\
d_m(x^m_p,x^m_q)&\le A_m B + B\, d(F_m(x^m_p),F_m(x^m_q))
\notag\\
&\le 2 A_m B + B\, d(x_p,x_q)    \qquad 
\text{for $m\ge m_1$ large enough.}
\label{dm2am}
\end{align}
There is $m_2\ge m_1$ such that 
\begin{equation}\label{m2l}
\forall m\ge m_2 \qquad A_m <\frac \de{8B}.
\end{equation}
If $d(x_p,x_q)<\tfrac 1{2B}\de$
then
\begin{align*}
2A_m B+ B \, d(x_p,x_q) 
%&<  2A_m B + \tfrac 12 \de
%\\
&< \tfrac 14 \de +\tfrac 12 \de < \de.
\end{align*}
From \eqref{dm2am} we obtain $d(x^m_p,x^m_q)<\de$ for $m\ge m_2$.
Then \eqref{equic} implies that 
\begin{equation}\label{fmxmp}
\forall m\ge m_2\qquad
|f_m(x^m_p)-f_m(x^m_q)|<\e.
\end{equation}
From~\eqref{akkn} and~\eqref{fmxmp} we get that
\begin{align*}
|f(x_p)-f(x_q)| &=|a_p-a_q| 
\\
&=\lim_k |f_{m(k,k)}(x^{m(k,k)}_p)-f_{m(k,k)}(x^{m(k,k)}_q)|
\\
&\le \e.
\end{align*}
Therefore $f|_\D$ is uniformly continuous
and thus we can extend $f$ to $M$ uniquely by continuity.
In particular $f$ is continuous on $M$.

Now we prove that $\lim_k f_{m(k,k)} = f$ uniformly on compact sets.
Suppose it is false. Then there is a compact subset $\K\subset M$ and $\th>0$ and 
a subsequence $m_k$ of $\{m(k,k)\}_{k\in \na}$ such that 
$$
\forall k\in\na\qquad
\sup_{F_{m_k}(y)\in\K} \big| f_{m_k}(y) -f(F_{m_k}(y))\big|>\th.
$$
Then there are $y_k\in M_{m_k}$ such that $F_{m_k}(y_k)\in\K$ and
\begin{equation}\label{e12}
\forall k\in\na \qquad
|f_{m_k}(y_k)-f(F_{m_k}(y_k))|>\th.
\end{equation}
Since $\K$ is compact by extracting a subsequence of $m_k$ we can assume that
$y:=\lim_k F_{m_k}(y_k)$ exists. By the continuity of $f$ we have that 
\begin{equation}\label{e11}
\lim_k |f(y)- f(F_{m_k}(y_k))|=0.
\end{equation}
Let $x_{n}\in\D$ be such that $\lim_n x_{n} = y$. By the continuity of $f$
we have that
 \begin{equation}\label{xn1n}
 \lim_n|f(x_n)-f(y)|=0.
 \end{equation}
By the construction above we have that
\begin{equation}\label{eqan}
\lim_k f_{m_k}(x^{m_k}_n) = f(x_n)= a_n.
\end{equation}
From~\eqref{1a} we have that 
\begin{align*}
d_{m_k}(y_k,x^{m_k}_n) 
&\le A_{m_k} B + B\, d(F_{m_k}(y_k), F_{m_k}(x_n^{m_k})).
\end{align*}
In the construction above we have that 
$\lim_k F_{m_k}(x_n^{m_k})=x_n$, $\lim_k F_{m_k}(y_k)=y$, therefore
$$
\limsup_k d_{m_k}(y_k,x^{m_k}_n) \le B \; d(y, x_n)\xrightarrow{\;n\;} 0.
$$
By the equicontinuity of $\{f_{m_k}\}_{k\in\na}$ the last inequality implies that
\begin{equation}\label{eq15}
\lim_n \limsup_k |f_{m_k}(y_k)-f_{m_k}(x^{m_k}_n)|=0.
\end{equation}

We have that
\begin{align*}
|f_{m_k}(y_k) - f(F_{m_k}(y_k))|
\le | f_{m_k}(&y_k) - f_{m_k}(x^{m_k}_n)|+|f_{m_k}(x^{m_k}_n) - f(x_n)|+
\\
&+|f(x_n) - f(y)| +
|f(y)-f(F_{m_k}(y_k))|.
\end{align*}

Using~\eqref{eqan}, \eqref{e11},
\begin{align*}
\limsup_k |f_{m_k}(y_k) &- f(F_{m_k}(y_k))|\le
\\
&\le \limsup_k  | f_{m_k}(y_k) - f_{m_k}(x^{m_k}_n)| + 0
+|f(x_n) - f(y)| + 0.
\end{align*}
Taking $\lim_n$ in the right hand side and using \eqref{eq15} and \eqref{xn1n}
we obtain
$$
\limsup_k |f_{m_k}(y_k) - f(F_{m_k}(y_k)) |=0,
$$
which contradicts \eqref{e12}.

\end{proof}

\begin{Proposition}\label{limlip}\quad

Suppose that $\lim_n (M_n,d_n,F_n)=(M,d)$, $(M,d)$ is separable, 
 $f_n:(M_n,d_n)\to\re$ is equi-Lipschitz
and $\lim_n f_n = f$ uniformly compact sets.
Then $f$ is Lipschitz.
\end{Proposition}

\begin{proof}
By Proposition~\ref{Paa}, the function $f$ is continuous.
Given $x,y\in M$, let $\K\subset M$ be a compact set such that  
 $x,y\in \interior\K$.  Let $x_n,y_n\in M_n$ be such that 
$\lim_n F_n(x_n)=x$, $\lim_n F_n(y_n)=y$ and $F_n(x_n),F_n(y_n)\in\K$.
We have that 
\begin{align}
|f(x)-f(y)| \le 
|f(x)&-f(F_n(x_n))|+|f(F_n(x_n))-f_n(x_n)|+|f_n(x_n)-f_n(y_n)|+
\notag\\
&+|f(y_n)-f(F_n(y_n))|+|f(F(y_n))-f(y)|.
\label{1st}
\end{align}
By the continuity of $f$ and the uniform convergence~\eqref{ucs} of $f_n$ on $\K$,
only the third term in the right hand side may not converge to zero.
Since the family $\{f_n\}$ is equi-Lipschitz, for the third term we have 
\begin{align}
|f_n(x_n)-f_n(y_n)| &\le Q \, d_n(x_n,y_n) 
\notag\\
&\le Q \, B \, d(F_n(x_n),F_n(y_n)) + Q \, B A_n
\qquad\text{using~\eqref{1a}.}
\label{3rd}
\end{align} 
Letting $n\to\infty$ in \eqref{1st} and \eqref{3rd} we get
$$
|f(x)-f(y)|\le QB\, d(x,y).
$$

\end{proof}

\color{black}

 \section{Comparison Theorem.}

We say that the hamiltonian $H:T^*M\to\re$ is {\it quadratic at infinity} if there is a riemannian metric on $M$, and $R>0$ such that if $x\in M$ and $|p|_x>R$ then
$$
H(x,p) = \tfrac 12 |p|_x^2+ p(\xi(x)) +V(x),
$$
where $\xi(x)$ is a smooth vector field on $M$ and 
$V:M\to\re$ is a smooth function.

\begin{Theorem}[Comparison Theorem]\label{tcomp}\quad

Suppose that $H:T^*M\to\re$ is quadratic at infinity. 
Let $H_\e:T^*\hM\to\re$ be 
$$
H_\e(x,p)=\hH(x,\tfrac 1\e p).
$$
Let  $\Om\subset ]0,T[\times \hM$ be a compact set.
Consider
\begin{equation}\label{HJ2}
\partial_t u +H_\e(x,\partial_x u) = 0   \qquad x\in \hM.
\end{equation}
Suppose that $u$ is a Lipschitz viscosity subsolution and 
$v$ is a Lipschitz viscosity supersolution of 
\eqref{HJ2} such that $u\le v$ on $\partial \Om$. 
 Then 
$u\le v$ on $\Om$.

\end{Theorem}

To illustrate the proof we show it in case $u$ and $v$ are differentiable.
If there is a point in $\Om$ where $u>v$ then there is $\la>0$ such that 
\begin{equation}\label{mxpt}
\sup_\Om (u-v-\la t)>0.
\end{equation}
Let $(x_0,t_0)$ be an (interior) point where the supremum \eqref{mxpt} is attained.
We have that
\begin{equation}\label{subsup}
\begin{aligned}\partial_t u(x_0,t_0) + H_\e(x_0, \partial_x u(x_0,t_0))&\le 0,
 \\
\partial_t v(x_0,t_0) + H_\e(x_0, \partial_x v(x_0,t_0))&\ge 0.
\end{aligned}
\end{equation}

\begin{gather}
\partial_x u(x_0,t_0)= \partial_x v(x_0,t_0),
\label{29}
\\
\partial_t u(x_0,t_0) -\partial_t v(x_0,t_0) = \la >0.
\label{30}
\end{gather}
Substracting the equations in \eqref{subsup} we get a contradiction with
 equations \eqref{29}, \eqref{30}.

\begin{proof}

Suppose there is a point $(x_0,t_0)$ in $\Om$ where $u>v$.
Then there is $\la>0$ such that 
\begin{equation}\label{maxpt}
\sup_\Om (u-v-2\la t) >0.
\end{equation}
%Let $(x_0,t_0)\in \Om$ be a maximizing point of \eqref{maxpt}.
%Observe that $(x_0,t_0)$ must be an interior point of $\Om$ because
%$u\le v $ on $\partial \Om$.
%Fix a coordinate chart of $N$ on a neighborhood of $(x_0,t_0)$.
%Let $A$ be a compact neighborhood of $(x_0,t_0)$ contained in the chart.
For $\de>0$ small, define $f_\de:\Om\times \Om\to \re$ by
$$
f_\de(x,t,y,s):= u(x,t)-v(y,s)-\la(t+s) 
-\tfrac 1\de\big(|t-s|^2+d(x,y)^2\big).
$$
Let $(x_\de,t_\de,y_\de,s_\de)$ be a maximizing point of $f_\de$ in $\Om\times \Om$.

We have that 
\begin{equation}\label{fm0}
f_\de(x_\de,t_\de,y_\de,s_\de)\ge f_\de(x_0,t_0,x_0,t_0)>0.
\end{equation}
\begin{align}
\tfrac 1\de \big(|t_\de-s_\de|^2+d(x_\de,y_\de)^2\big)
&< u(x_\de,t_\de)-v(y_\de,s_\de)-\la(s_\de+t_\de)
\notag\\
&\le \sup_\Om |u|+\sup_\Om|v| + \sup_\Om 2\la|t| =:Q
\notag\\
|t_\de-s_\de|^2+d(x_\de,y_\de)^2
&\le \de \,Q
\notag\\
\max\{|s_\de-t_\de|,\,d(x_\de,y_\de)\} &\le Q_1\, \sqrt{\de}, \qquad \qquad Q_1:=\sqrt{Q}.
\label{dxy}
\end{align}
Observe that if $(x,t)\in\partial \Om$
then $u(x,t)\le v(x,t)$ and hence, for $\tau:=\inf\{t : (x,t)\in\Om\}$,
\begin{equation}\label{fbound}
\forall(x,t)\in\partial\Om\qquad f_\de(x,t,x,t) \le - 2\la t\le -2\la \tau.
\end{equation}

We show now that both $(x_\de,t_\de)$ and $(y_\de,s_\de)$ are interior points of $\Om$
if $\de$ is small enough.
Suppose on the contrary that $(x_\de,t_\de)\in\partial\Om$.
Let $K$ be a Lipschitz constant for $v$. Using \eqref{fbound},
we have that 
\begin{align*}
f_\de(x_\de,t_\de,y_\de,s_\de)
&= f_\de(x_\de,t_\de,y_\de,s_\de) -f(x_\de,t_\de,x_\de,t_\de)+f(x_\de,t_\de,x_\de,t_\de)
\\
&\le\big(-v(y_\de,s_\de)+v(x_\de,t_\de)-\la(s_\de-t_\de)\big) -2\la\tau
\\
&\le K\big(d(x_\de,y_\de)+|t_\de-s_\de|\big)+\la |t_\de-s_\de|-2\la \tau.
\end{align*}
Then, using~\eqref{dxy}, there is $\de_0>0$ such that 
$$
0<\de<\de_0 \quad\&\quad (x_\de,t_\de)\in\partial \Om
\quad \then \quad f_\de(x_\de,t_\de,y_\de,s_\de)<0.
$$
%if $0<\de<\de_0$ and $(x_\de,t_\de)\in\partial \Om$
%then $f_\de(x_\de,t_\de,y_\de,s_\de)<0$.
This contradicts \eqref{fm0}, therefore $(x_\de,t_\de)\notin\partial \Om$
and similarly $(y_\de,s_\de)\notin\partial \Om$ if $0<\de<\de_0$.

Define $\phi:]0,T]\times N\to\re$ by
\begin{gather*}
u(x,t)-\phi(x,t) :=f_\de(x,t,y_\de,s_\de),
\\
\phi(x,t)= v(y_\de,s_\de)+\la(t+s_\de)
+\tfrac 1\de(|t-s_\de|^2+d(x,y_\de)^2).
\end{gather*}
Observe that $u-\phi$ has a local maximum at $(x_\de,t_\de)$, then
\begin{gather}
\partial_t\phi(x_\de,t_\de)+H_\e\big(x_\de,\partial_x\phi(x_\de,t_\de)\big)\le 0,
\notag\\
\la+ \tfrac 2\de (t_\de-s_\de)+H_\e\big (x_\de, \tfrac 2\de d(x_\de,y_\de)\, \nabla_x d(x_\de,y_\de)\big)\le 0.
\label{35}
\end{gather}
Define $\psi:]0,T]\times N\to\re$ by
\begin{gather*}
v(y,s)-\psi(y,s) :=-f_\de(x_\de,t_\de,y,s),
\\
\psi(x,t)= u(x_\de,t_\de)-\la(t_\de+s)
-\tfrac 1\de(|t_\de-s|^2+d(x_\de,y)^2).
\end{gather*}
Then $v-\psi$ has a local minimum at $(y_\de,s_\de)$. Since $v$ is a supersolution we have that 
\begin{gather}
\partial_s\psi(y_\de,s_\de)+H_\e\big(y_\de,\partial_y\psi(y_\de,s_\de)\big)\ge 0,
\notag\\
-\la- \tfrac 2\de (s_\de-t_\de)+H_\e\big (y_\de, -\tfrac2\de d(x_\de,y_\de)\, \nabla_y d(x_\de,y_\de)\big)\ge 0.
\label{36}
\end{gather}
Subtracting
\eqref{35}$-$\eqref{36} we obtain
\begin{equation}\label{37}
2\la + H_\e\big (x_\de, \tfrac 2\de d(x_\de,y_\de)\, \nabla_x d(x_\de,y_\de)\big)
-H_\e\big (y_\de, -\tfrac 2\de d(x_\de,y_\de)\, \nabla_y d(x_\de,y_\de)\big)
\le 0.
\end{equation}
Since $\la>0$, if we can prove that 
\begin{equation}\label{38}
H_\e\big (x_\de, \tfrac 2\de d(x_\de,y_\de)\, \nabla_x d(x_\de,y_\de)\big)
-H_\e\big (y_\de, -\tfrac 2\de d(x_\de,y_\de)\, \nabla_y d(x_\de,y_\de)\big)
\end{equation}
is small for $\de$ small, we get a contradiction with \eqref{37}.

Since by \eqref{dxy} $d(x_\de,y_\de)$ is very small we have that 
$$
\nabla_{y_{\de}}d(x_\de,y_\de)=\ga_{xy}'(d(x_\de,y_\de)),
$$
 where $\ga_{xy}(t)$ is the 
unit speed geodesic going from $x_\de$ to $y_\de$.

Consider the local parametrization of $\hM$ given by the exponential map
\begin{equation}\label{exp}
\exp_{x_\de}:T_{x_\de}\hM\to \hM.
\end{equation}
 In these coordinates we have that
\begin{alignat*}{2}
\tfrac 2\de\, d(x_\de,y_\de)\, \nabla_y d(x_\de,y_\de) 
&= \tfrac 2\de(y_\de-x_\de),
\qquad
&& x_\de= 0\in T_{x_\de}\hM,
\\
\tfrac 2\de\, d(x_\de,y_\de)\, \nabla_x d(x_\de,y_\de) 
&= -\tfrac 2\de(y_\de-x_\de),
&&x_\de= 0\in T_{x_\de}\hM.
\end{alignat*}
In these coordinates \eqref{38} becomes
\begin{equation}\label{40}
H_\e(x_\de, p_\de) - H_\e(y_\de, p_\de),
\qquad p_\de := -\tfrac 2\de (y_\de-x_\de).
\end{equation}
Since $H$ is quadratic at infinity, there is $A>0$ such that
\begin{equation}\label{41}
|H_\e(x,p)-H_\e(y,p)|\le A \;(1+\tfrac 1{\e^2}|p|_x^2)\; d(x,y),
\end{equation}
and this constant $A$ is uniform for the unit ball 
$B(0,1)\subset T_x\hM$ in the coordinates \eqref{exp}
{\sl for every} $x\in\hM$.
When $\de\to 0$ the points $x_\de$ and $y_\de$ move,
but inequalities \eqref{dxy} and \eqref{41}
remain valid with the same constants $Q_1$ and $A$.

We need to bound the distance from $(x_\de,t_\de)$ to $\partial \Om$.
If $(z,r)\in\partial \Om$ by~\eqref{fbound} we have that 
$$
0 \le f(x_\de,t_\de,y_\de,s_\de)-f(z,r,z,r)+f(z,r,z,r),
$$
\begin{align*}
0\le u(x_\de,t_\de)&-v(y_\de,s_\de)-\la(t_\de+s_\de)
-\tfrac 1{\de}(|t_\de-s_\de|^2+d(x_\de,y_\de)^2)-
\\
&-u(z,r)+v(z,r)+\la 2r
-2\la\tau,
\end{align*}
\begin{align*}
0\le  K\big( d[(x_\de,t_\de),(z,r)]+d[(y_\de,s_\de),(z,r)]\big)
+|\la|(|t_\de-r|+|s_\de-r|) -2\la\tau.
\end{align*}
From \eqref{dxy}
\begin{align*}
d[(y_\de,s_\de),(z,r)]&\le  d[(x_\de,t_\de),(z,r)]+2 Q_1 \sqrt{\de},
\\
|s_\de-r| &\le |t_\de-r|+Q_1\sqrt{\de}.
\end{align*}
Therefore
$$
2\la\tau
-(2K+\la)\,2Q_1\,\sqrt{\de}
\le 2K d[(x_\de,t_\de),(z,r)]+2 \la |t_\de-r|.
$$
Then for $\de$ small enough
$$
\la \tau \le 2(K+\la)\,d[(x_\de,t_\de),(z,r)].
$$
This implies that there are $a>0$ and $\de_1>0$  such that
$$
\forall \de\in]0,\de_1[
\qquad
 d((x_\de,t_\de),\partial\Om) >a.
$$
This and \eqref{dxy} imply that for $\de$ small enough
$(x_\de,s_\de)\in\interior\Om$.

Since $(x_\de,t_\de,y_\de,s_\de)$ is a maximum on $\Om\times\Om$,
$v$ is Lipschitz and $(x_\de,s_\de)\in\interior\Om$, we have that 
\begin{align*}
0&\le f_\de(x_\de,t_\de,y_\de,s_\de)- f_\de(x_\de,t_\de, x_\de,s_\de)
\\
0&\le
\big( v(x_\de,s_\de)-v(y_\de,s_\de)\big)
-\tfrac 1\de \, d(x_\de,y_\de)^2
\end{align*}
\begin{align}
\tfrac 1\de  \, d(x_\de,y_\de)^2
&\le K\, d(x_\de,y_\de)
\notag\\
d(x_\de,y_\de) &\le K\, \de.
\label{42}
\end{align}

Using  \eqref{40}, \eqref{41}, \eqref{42}  we obtain
\begin{equation*}
|H_\e(x_\de,p_\de)-H_\e(y_\de,p_\de)|\le
A(1+\tfrac 1{\e^2}  \, 4 K^2) K \de
\xrightarrow{\;\;\de\downarrow 0\;\;} 0.
\end{equation*}
Since~\eqref{40} is equal to \eqref{38}, we get that   \eqref{38} is arbitrarily small for $\de$ small enough
and this contradicts \eqref{37}.

\end{proof}

\begin{Corollary}\label{ccomp}\quad

If $H$ is quadratic at infinity, $\Om\subset ]0,T[\times\hM$ is a compact set, $u$ is a Lipschitz viscosity 
subsolution of \eqref{HJ2} and $v$ is a Lipschitz viscosity supersolution of \eqref{HJ2}
then 
$$
\max_\Om (u-v) \le \max_{\partial\Om}(u-v).
$$
\end{Corollary}

\begin{proof}\quad

Observe that if $a\in\re$ then $v+a$ is also a supersolution of \eqref{HJ2}.
Apply this to 
\linebreak
$a=\max_{\partial \Om}(u-v)$. We get that
on $\partial \Om$, $u\le v+ a$, then from Theorem~\ref{tcomp},
$u\le v+a$ on $\Om$, this is 
$$
\max_\Om (u-v) \le a = \max_{\partial\Om}(u-v).
$$
\end{proof}

\bigskip

\section{Proof with test functions.}

Here we prove Theorem~\ref{T1}.
By Proposition~\ref{eqlip} and Proposition~\ref{Paa} in order to prove that $\lim_\e v^\e=u$
it is enough to prove that there is a unique possible limit of subsequences $v^{\e_n}$.
Since the equation~\eqref{hhj} has a unique solution, the following proposition finishes 
the proof of Theorem~\ref{T1}.

\begin{Proposition}\label{Plim}\quad

If a subsequence $u^{\e_n}$ of the family  $u^\e$ in~\eqref{uef} 
(of solutions to~\eqref{eve}) converges uniformly on compact subsets
to a function $u=\lim_n u^{\e_n}$, then $u$ satisfies in the viscosity sense the equation
\begin{equation}\label{hjj2}
\begin{gathered}
\partial_t u +\oH(\partial_x u) = 0,\\
u(x,0)=f(x).
\end{gathered}
\end{equation}

\end{Proposition}

\begin{Lemma}\quad\label{Llim}

It is enough to prove Proposition~\ref{Plim} for hamiltonians quadratic at infinity.
\end{Lemma}

\begin{proof}\quad

By Proposition~\ref{eqlip} there is a uniform Lipschitz constant $Q$ for all the solutions $u^\e$.
By Rademacher Theorem $\partial_y u^\e$ is defined almost everywhere and it is a weak derivative of $u^\e$.
Observe that 
\begin{align*}
Q\ge \lV \partial_y u^\e\rV_\e 
= \sup_{|v|_\e=1} |\partial_y u^\e\cdot v|
=\sup_{|v|_1=\frac 1\e}  |\partial_y u^\e\cdot v|
=\tfrac 1\e\; \lV\partial_y u^\e\rV_1.
\end{align*}
Thus in equation~\eqref{eve} 
\begin{equation}
\tag{\ref{eve}} \partial_t u^\e + \hH(y,\tfrac 1\e\partial_y u^\e) =0,
\end{equation}
we only use the Hamiltonian $H$ on co-vectors $(y,p)\in T^*M$ with $\lv p\rv_y \le Q$.

By Proposition~\ref{limlip} the limit function $u$ has some Lipschitz constant $QB$,
with $B>1$ from~\eqref{1a}.
Thus the effective hamiltonian $\ov H$ in equation \eqref{hjj2} 
for $u=\lim u^{\e_n}$ is only used on co-vectors 
$c\in H_1(M,\re)^*=H^1(M,\re)$ with norm $|c|\le QB$.

On the other hand, by Proposition~\ref{limlip}, the function $f$ is Lipschitz. 
Since Mather's $\a$ and $\be$ functions $\ov H$ and $\ov L$ 
are convex and superlinear, by Proposition~\ref{eqlip}
there is $Q_1>Q>0$ such that any solution of  the problem~\eqref{hjj2}
has Lipschitz constant $Q_1$.

 The effective hamiltonian $\ov H$
 is Mather's alpha function which satisfies (see~\cite{CIPP} Cor.~1):
$$
\oH(c) = \inf_{[\eta]=c}\sup_{x\in M} H(x,\eta(x)).
$$
Therefore
$$
H_1 = H \quad\text{on}\quad [H\le \oH(c)] \quad\then \quad
\ov H_1(c)=\ov H(c).
$$
Let $h_0:=\sup_{|c|\le Q_1 B} \oH(c)$.
Then if $H_1 = H$ on
$[H\le h_0]$, any limit function $u=\lim_n u^{\e_n}$
satisfies the problem \eqref{hjj2} if and only if it satisfies
the problem with $H_1$:
\begin{equation}\label{81ov}
\begin{gathered}
\partial_t u +\oH_1(\partial_x u) = 0,\\
u(x,0)=f(x).
\end{gathered}
\end{equation}
Moreover, the solutions of problems~\eqref{hjj2} and \eqref{81ov} are equal.

By the superlinearity of $H$ there is $R_0> Q_1$ such that 
$[H\le h_0]\subset [|p|_x\le R_0]$.
Therefore we can replace $H$ with a hamiltonian $H_1$ quadratic at infinity
such that $H_1=H$ on $|p|_x\le R_0$ and we will have that the families $u^\e$, the limits $u$ 
and any solution of \eqref{hjj2}
will be the same for both hamiltonians.
For a construction of such hamiltonian quadratic at infinity see Proposition~18 in \cite{CIPP2}.

\end{proof}

\noindent{\bf Proof of Proposition~\ref{Plim}:}\quad

    By Lemma~\ref{Llim} we can assume that $H$ is quadratic at infinity, so we can apply
    Corollary~\ref{ccomp}.
   We prove that $u$ is a viscosity subsolution of \eqref{hjj2}.
   The proof for supersolution will be similar.
   By Lemma~\ref{l03}, it is enough to use Lipschitz test functions giving 
   strict local maxima.
   
  Let $\phi\in C^1(H_1(M,\re)\times ]0,+\infty[,\re)$
    be such that  $\phi$ is Lipschitz and 
    \begin{equation}\label{strict}
    \text{$u-\phi$ has a strict local maximum at $(y_0,t_0)$.}
    \end{equation}
    We want to show that
    \begin{equation}\label{18}
    \partial_t\phi(y_0,t_0)+\oH(D_x\phi(y_0,t_0))\le 0.
    \end{equation}
  
    Assume that \eqref{18} is not true, i.e. there is $\th>0$ such that
    \begin{equation}\label{eth}
    \partial_t\phi(y_0,t_0)+\oH(D_x\phi(y_0,t_0))=\th>0.
    \end{equation}

    Let $w:M\to\re$ be a viscosity solution to the cell problem
    \begin{equation}
    \label{cell2}
    H(x,g(P)+D w(x))=\oH(P), \qquad P:=D_x\phi(y_0,t_0)\in H^1(M,\re),
    \end{equation}
    and let $\hw = w\circ \pi$ be its lift to $\hM$.
    
    Define $\phi^\e:M_\e\times\re_+\to\re$, $M_\e=(\hM, d_\e)$,  by
    \begin{equation}\label{phie}
    \phi^\e(x,t):=\phi(F_\e(x),t)+\e\,\hw(x).
    \end{equation}
    
    By \eqref{lcell} the lift $\hw=w\circ\pi$
    is a viscosity solution of 
    \begin{equation}\label{lcell2}
    \hH(x, P\cdot DG(x) + D\hw(x))=\oH(P), \quad x\in\hM, 
    \quad P=D_x\phi(y_0,t_0)\in H^1(M,\re).
    \end{equation}
    Observe that in \eqref{lcell2} $P=D_x\phi(y_0,t_0):H_1(M,\re)=\re^k\to\re$ 
    is a linear functional on $H_1(M,\re)=\re^k$ and also on each tangent 
    space $T_z H_1(M,\re)$, even when $z\ne y_0$, this is
    $$
    P=D_x\phi(y_0,t_0)\in H^1(M,\re)=H_1(M,\re)^*.
    $$

 \bigskip
    
    \begin{claim} \label{cl1}
    For some $r,\e_0>0$ small and $\hH:=H\circ(\pi^*)^{-1}$,
    the lift of $H$,
    \begin{equation}\label{HJt2}
    \forall \e<\e_0 \qquad
    \partial_t\phi^\e+ \hH(x,\tfrac 1\e D_x\phi^\e)\ge \tfrac \th 2
    \qquad (supersolution)
    \end{equation}
    \qquad\text{ 
    in the viscosity sense in $(F_\e\times id)^{-1}(B_r(y_0,t_0))$.
    }
    \end{claim}
  
    {\bf Proof of the Claim:}
    
    Suppose that $\phi^\e-\psi^\e$ has a  minimum at $(x_1^\e,t_1)\in M_\e\times\re_+$
    with $F_\e(x_1^\e)\to y_1$ and $(y_1,t_1)$ near $(y_0,t_0)$. 
    We have to prove that
    \begin{equation}\label{caim}
    \partial_t\psi^\e(x_1^\e,t_1)+\hH\big(x_1^\e,\tfrac 1\e D_x\psi^\e(x_1^\e,t_1)\big)\ge \tfrac \th 2.
    \end{equation}
 
    We have that 
    $$
    \phi^\e-\psi^\e\ge \phi^\e(x_1^\e,t_1)-\psi^\e(x_1^\e,t_1).
    $$
    Using \eqref{phie}, this is
    $$
    \phi(F_\e(x),t)+\e\, \hw(x) -\psi^\e(x,t)
    \ge
    \phi(F_\e(x_1^\e),t_1)+\e\, \hw(x_1^\e)-\psi^\e(x_1^\e,t_1).
    $$
    Define $\eta:\hM\times\re_+\to\re$ by
    $$
    \eta(x,t):=\tfrac 1\e\,\big[\psi^\e(x,t)-\phi(F_\e(x),t)\big].
    $$
    Then $\hw-\eta$ has a local minimum at $(x_1^\e,t_1)$.
    
    By the viscosity property \eqref{supersol} on the lift  \eqref{lcell} 
    of  \eqref{cell2}, we have that
    $$
    \hH(x_1^\e, P\cdot DG(x_1^\e)+D_x\eta(x_1^\e,t_1))\ge \oH(P),
    $$
    where $P=D_x\phi(y_0,t_0)\in H^1(M,\re)$ and $\hH=H\circ (\pi^*)^{-1}$.
    Then by \eqref{eth},
    $$
    \partial_t\phi(y_0,t_0)+\hH\big(x_1^\e,P\cdot DG(x_1^\e)+D_x\eta(x_1^\e,t_1)\big)
    \ge \partial_t\phi(y_0,t_0)+\oH(P)=\th.
    $$
    Using\eqref{Fe}, observe that
    $$
    D_x\eta(x,t)=-D_x\phi(F_\e(x),t) \,DG(x) +\tfrac 1\e\,D_x\psi^\e(x,t).
    $$
    Thus
    $$
    \partial_t \phi_t(y_0,t_0)
    +\hH\big[x_1^\e,\,P\cdot DG(x_1^\e)-D_x\phi(F_\e(x^\e_1),t_1)\,DG(x_1^\e)
    +\tfrac 1\e\,D_x\psi^\e(x_1^\e,t_1)\big]
    \ge\th.
    $$
    
    Writing $y_1^\e=F_\e(x_1^\e)$
    and recalling that $P=D_x\phi(x_0,t_0)$, 
    we have that 
    \begin{align*}
    \lv \partial_t\phi(y_0,t_0)\right.&
    \left.-\;\partial_t\phi(y_1^\e,t_1)\rv
    +\partial_t\phi(y_1^\e,t_1)\;+\\
    &+\hH\big[x_1^\e,\,
    \big(D_x\phi(y_0,t_0)-D_x\phi(y_1^\e,t_1)\big)\,DG(x_1^\e)
    +\tfrac 1\e\,D\psi^\e(x_1^\e,t_1)\big]
    \ge \th.
    \end{align*}

      In \eqref{strict} we have that  $\phi\in C^1(H_1(M,\re)\times\re_+,\re)$
      and by \eqref{NG},  $\lV DG\rV$ is bounded. Thus, if $d[(y_0,t_0),(y_1^\e,t_1)]<r$ 
    is small enough, then
    $$
   \partial_t  \phi(y_1^\e,t_1)+\hH[x_1^\e,\tfrac 1\e\,D_x\psi^\e(x_1^\e,t_1)]\ge\frac\th 2.
    $$
    
    Since $\hw-\eta$ has a local minimum at $(x_1^\e,t_1)$ and $\hw$ 
    is time-independent, we have that
    $$
    \partial_t\eta(x_1^\e,t_1) =
    \tfrac 1\e\,\big[\partial_t\psi^\e(x_1^\e,t_1)
    -\partial_t\phi(F_\e(x_1^\e),t_1) \big]
    =0.
    $$
    This is 
    $
    \partial_t\psi^\e(x_1^\e,t_1)=\partial_t \phi(y_1^\e,t_1).
    $
    Therefore we get \eqref{caim}:
    $$
    \partial_t\psi^\e(x_1^\e,t_1) + \hH\big(x_1^\e,\tfrac 1\e\,D_x\psi^\e(x_1^\e,t_1)\big)
    \ge\tfrac\th 2.
    $$
    This proves the claim, i.e.
    $\phi^\e$ satisfies \eqref{HJt2}
    in the viscosity sense in a small neighborhood
    $(F_\e\times id)^{-1}(B_r(y_0,t_0))$.
   
    \hfill$\triangle$

    \medskip

   Let $x_0^\e\in M_\e$ be such that $\lim_\e F_\e(x_0^\e)= y_0$.
   By  Claim \ref{cl1}, the function $\phi^\e$ is a viscosity supersolution 
   of 
   \begin{equation}\label{HJe}
   \partial_t u+\hH(x,\tfrac 1\e\,D_xu)=0
   \end{equation}
   nearby $(x_0^\e,t_0)$. 
    \bigskip

    By the choice of $\phi$ in \eqref{strict}, $\phi$ is Lipschitz.
    The solution $w$ to the cell problem \eqref{cell2}, in $M$ is also Lipschitz.
    By Lemma~\ref{lLip}, the map
    $F_\e$ is Lipschitz on $(M_\e,d_\e)$.
    Then by \eqref{phie} $\phi^\e$ is Lipschitz on $(M_\e,d_\e)$.
    By Proposition~\ref{eqlip} the functions $u^\e$ are Lipschitz on $(M_\e,d_\e)$.
    
    By Corollary~\ref{c1.3}, if $0<r<t_0$ the set $(F_\e\times id)^{-1}(\ov{B_r(y_0,t_0)})$ is compact.
    Since $\phi^\e$ is a supersolution and $u^\e$ is a subsolution
    of \eqref{HJe} and $F_\e(x^\e_0)\xrightarrow{\e} y_0$; by Corollary~\ref{ccomp} 
    for $\e$ small enough
    we have that
    \begin{equation}\label{uhb}
    u^\e(x_0^\e,t_0)-\phi^\e(x_0^\e,t_0)
    \le \sup_{\partial (F_\e\times id)^{-1}(B_r(y_0,t_0))}
    (u^\e-\phi^\e).
    \end{equation}
        
    Let $(z^\e,s^\e)\in \partial (F_\e \times id)^{-1}(B_r(y_0,t_0))$  be such that 
    $$
    u^\e(x_0^\e,t_0)-\phi^\e(x_0^\e,t_0)
    \le u^\e(z^\e,s^\e)-\phi^\e(z^\e,s^\e).
    $$
    Observe that from \eqref{phie}, 
    if $F_\e(x^\e)\to y$ and $t^\e\to t$ then $\phi^\e(x^\e,t^\e)\to \phi(y,t)$.
    Let $\e_i\to 0$ be a sequence such that
    the following holds:
    \begin{enumerate}[(a)]
    \item $\lim_i u^{\e_i}(x_0^{\e_i},t_0)= u(y_0,t_0)$,
    \item the limit 
    $(z_0,s_0):=\lim_i (F_{\e_i}(z^{\e_i}),s^{\e_i})$
    exists,
    \item $\lim_i u^{\e_i}(z^{\e_i},s^{\e_i})= u(z_0,s_0)$.
    \end{enumerate}
   Then we have that 
    $d((y_0,t_0),(z_0,s_0))=r$ and
    $$
    u(y_0,t_0)-\phi(y_0,t_0)\le u(z_0,s_0)-\phi(z_0,s_0).
    $$

   This contradicts \eqref{strict}, the  strict local maximum property of $(y_0,t_0)$. 
   Thus \eqref{eth} is impossible, and hence we get \eqref{18}:
   $$
   \partial_t\phi (y_0,t_0)+\oH(D_x\phi (y_0,t_0))\le 0
   $$
   whenever $\phi\in C^1$, $\phi$ is  Lipschitz, and $u-\phi$ has a strict local maximum at 
   a point $(y_0,t_0)$. Therefore $u$ is a viscosity subsolution 
   of \eqref{hjj2}.
   
   The proof that $u$ is a viscosity supersolution of  \eqref{hjj2} is similar.
   
    \qed

\section{Uniqueness.}
\label{Suniqueness}

We prove here the uniqueness of Lipschitz viscosity solutions
to the initial value problem of the evolutive Hamilton-Jacobi equation
on the abelian cover $\hM$.

  \begin{Lemma}\label{LT}\quad
  
  Suppose that $u:\hM\times[0,T]\to\re$ is a continuous function which is 
  a viscosity subsolution of \eqref{HJ2} on the  open interval $t\in]0,T[$.
  If $\phi$ is $C^1$ and $u-\phi$ attains a local maximum at a point 
  $(x_0,T)$, then
  \begin{equation}\label{81}
  \partial_t\phi(x_0,T) + H(x_0,\partial_x\phi(x_0,T))\le 0.
  \end{equation}
  A corresponding statement holds for supersolutions.
  \end{Lemma}
  
  \begin{proof}\quad
  
  Replacing $\phi$ by $\phi(x,t) + d(x,x_0)^2+|t-T|^2$ we can assume that 
  $u-\phi$ attains a strict local maximum at $(x_0,T)$.
  For $\de>0$ let
  $$
  \phi_\de(x,t) :=\phi(x,t)+\frac\de{T-t}.
  $$
  Then the function $u-\phi_\de$ attains a local maximum at a point $(x_\de,t_\de)$ 
  with
  $$
  t_\de<T, 
  \qquad 
  (x_\de,t_\de) \longrightarrow (x_0,T) 
  \quad\text{as } \de \to 0^+.
  $$
  Since $u$ is a viscosity subsolution on $t<T$, we have that 
  \begin{align*}
  \partial_t\phi(x_\de,t_\de)
  +H(x_\de,\partial_x \phi(x_\de,t_\de))
  &=
  \partial_t\phi_\de(x_\de,t_\de)
  +H(x_\de,\partial_x \phi_\de(x_\de,t_\de)) -\frac{\de}{(T-t_\de)^2}\le 0.
  \end{align*}
  Letting $\de\to 0$ we obtain \eqref{81}.
  
  \end{proof}

 % \bigskip
  \begin{Proposition}\label{Pun}\quad
  
  Suppose that $H:T^*M\to\re$ is quadratic at infinity.
  
  Let $H_\e:T^*\hM\to\re$ be $H_\e(x,p)=\hH(x,\tfrac 1\e p)$.
  Consider the equation
  \begin{equation}\label{HJ5}
  \partial_t u+H_\e(x,\partial_xu)=0.
  \end{equation}
  Suppose that  $u$ is a Lipschitz viscosity subsolution of \eqref{HJ5} and 
  $v$ is a Lipschitz viscosity supersolution of \eqref{HJ5}.
  If $u\le v$ on $\hM\times\{0\}$ then $u\le v$.
 \end{Proposition}
  
  \begin{proof}\quad

  Assume that $u(x,0)\le v(x,0)$ for all $x\in\hM$.
  Suppose by contraposition that 
  there is a point in $(x_0,t_0)\in \hM\times [0,T]$ where $u>v$.
  Then there is $\la>0$ such that 
  $$
  (u-v-2\la\,t)(x_0,t_0)=:\si>0.
  $$
  
   The function $g(x):= d(x,x_0)$ is not differentiable at $x_0$
  and the points in the cut locus of $x_0$. It is Lipschitz and hence weakly differentiable
  and $\lV \nabla_x d(x,x_0)\rV=1$ almost everywhere. Let $f:\hM\to\re$ be a smooth
  function such that 
  \begin{gather}\label{fx}
  d(x,x_0)>1\quad  \then\quad  f(x) > \tfrac 12 \, d(x,x_0), 
  \\
  \forall x\in \hM  \qquad   \lV\nabla_x f\rV\le 2,
  \label{nf}
  \\
  f(x_0)=0.
  \notag
  \end{gather}

    Let $K$ be a Lipschitz constant for $u$ and for $v$. Then
  \begin{align}
  u(x_\de,t_\de)-u(y_\de,s_\de)
  &= u(x_\de,t_\de) -u(x_\de,0)+u(x_\de,0)
  -u(y_\de,0)
 \notag \\
  &\hskip 1.9cm +u(y_\de,0)-v(y_\de,0)
  +v(y_\de,0)-v(y_\de,s_\de)
  \notag
  \\
  &\le K|t_\de| + K d(x_\de,y_\de)  + 0 +K|s_\de|
  \label{du}
  \end{align}

  For $0<\de<1$ small, define $f_\de:(\hM\times [0,T])^2\to\re$ by
  \begin{equation*}\label{fde}
  f_\de(x,t,y,s):= u(x,t)-v(y,s)-\la(t+s)-\tfrac 1{\de^2}\big(|t-s|^2+ d(x,y)^2\big)
  -\de\, \big( f(x)^2+f(y)^2\big).
  \end{equation*}
  Since $u$ and $v$ have linear growth and by~\eqref{fx},
  $f(x)^2$ has quadratic growth;
  there is a point $(x_\de,t_\de,y_\de,s_\de)$ which maximizes $f_\de$ in 
  $\hM\times[0,T]$.

  We have that 
  \begin{equation}\label{si}
  f_\de(x_\de,t_\de,y_\de,s_\de)\ge f_\de(x_0,t_0,x_0,t_0)=\si>0.
  \end{equation}

  Using \eqref{du}, we have that 
  \begin{align}
  \tfrac 1{\de^2}\big(|t_\de-s_\de|^2+d(x_\de,y_\de)^2\big)
  &<
  u(x_\de,t_\de)-v(y_\de,s_\de)-\la(s_\de+t_\de)
  \notag\\
  &\le K|t_\de| + K d(x_\de,y_\de)  +K|s_\de| + 0
  \notag\\
  &\le K\,d(x_\de,y_\de) + 2KT.
  \label{dtdx}
  \end{align}
  Then
  $$
  \tfrac 1{\de^2}d(x_\de,y_\de)^2 \le K d(x_\de,y_\de) +2 K T.
  $$
  We have that $z:=d(x_\de,y_\de)$ satisfies the inequality
  $$
  z^2 -K \de^2 z -2KT\de^2 \le 0.
  $$
  Then $z$ must be smaller than the larger root of this quadratic polynomial, i.e. 
  \begin{equation}\label{dq}
  d(x_\de,y_\de)=z \le \tfrac 12 \big( K\de^2 +\sqrt{K^2\de^4+8KT\de^2}\big)
  \le Q_0\, \de,
  \end{equation}
  where $Q_0=Q_0(T)$.
  From \eqref{dtdx}  using \eqref{dq} and $\de\le 1$ we also obtain
  \begin{equation}\label{dt}
  |t_\de-s_\de|\le Q_1 \, \de, 
  \end{equation}
  with some $Q_1=Q_1(T) \ge Q_0(T)$.

  Observe that if $\de$ is small enough, by \eqref{dq}
  the point 
  $y_\de$ is not in the cut locus of $x_\de$ and viceversa.
  Therefore the function $d(x,y)^2$ is differentiable at $(x_\de,y_\de)$
  with partial derivatives
  $\nabla_x d(x_\de,y_\de)= 2\, d(x_\de,y_\de) \nabla_x d(x_\de,y_\de)$,
  $\lV \nabla_x d(x,y)\rV =1$.

  We show that the maximum of $f_\de$ is not attained at 
  the initial time $t=0$.
  From \eqref{si} we have that 
  \begin{align*}
  \si &\le f_\de(x_\de,t_\de,y_\de,s_\de) 
  \le u(x_\de,t_\de)-v(y_\de,s_\de)
  &&
  \\
  &\le K |t_\de| + KQ_0\,\de + K|s_\de|
  &&\text{using \eqref{du} and \eqref{dq},}
  \\
  &\le 2 K \, |t_\de| + KQ_1 \de + K Q_0\,\de
  &&\text{using \eqref{dt}.}
  \end{align*}
  A similar inequality holds for $|s_\de|$.
  For $\de>0$ sufficiently small, this implies that there is $\mu>0$ such that 
  \begin{equation}\label{87}
  t_\de> \mu>0 
  \qquad \text{and}\qquad
  s_\de>\mu>0.
  \end{equation}
  We could still have $t_\de=T$ or $s_\de=T$.

  Define $\phi:\hM\times[0,T]\to\re $ by
  \begin{align*}
  u(x,t)&-\phi(x,t):=
  f_\de(x,t,y_\de,s_\de)
  \\
  &=
  u(x,t)-v(y_\de,s_\de) -\la(t+s_\de)
  -\tfrac 1{\de^2} \big(|t-s_\de|^2 +d(x,y_\de)^2\big)
  -\de\, \big(f(x)^2+f(y_\de)^2\big).
  \end{align*}
  Observe that $\phi$ is $C^1$.
  We have that $u-\phi$ attains a local maximum in  $\hM\times [0,T]$ at $(x_\de,t_\de)$.
  If $(x_\de,t_\de)$  is an interior point, since $u$ is a viscosity subsolution
  we have that 
  \begin{equation}\label{88}
  \partial_t\phi(x_\de,t_\de) + H(x_\de,\partial_x \phi(x_\de,t_\de)) \le 0.
  \end{equation}
  By \eqref{87}, $t_\de>0$ and by Lemma~\ref{LT}
  equation~\eqref{88} also holds if $t_\de=T$.
  From \eqref{88} we get that 
  \begin{align}\label{89}
  \la +\tfrac 2{\de^2}(t_\de-s_\de)
  +H\big(x_\de,\tfrac 2{\de^2}\, d(x_\de,y_\de) \, \nabla_x d(x_\de,y_\de)
  + 2 \de f(x_\de) \nabla f (x_\de)\big)
  \le 0.
  \end{align}

 Define $\psi:\hM\times[0,T]\to\re$ by
 \begin{align*}
 v(y,s)&-\psi(y,s) := -f_\de(x_\de,t_\de,y,s)
 \\
 &= v(y,s) -u(x_\de,t_\de) +\la (s+t_\de)
  + \tfrac 1{\de^2}\big(|s-t_\de|^2+d(x_\de,y)^2\big)
  +\de\,\big(f(x_\de)^2+f(y)^2\big).
 \end{align*}
 Then $\psi(y,s)$ has a local minimum at 
 $(y_\de,s_\de)$. Since $v$ is a viscosity supersolution we have that 
 \begin{equation*}
 \partial_s \psi(y_\de,s_\de)
 +H\big(y_\de,\partial_y\psi(y_\de,s_\de)\big)\ge 0,
 \end{equation*}
 and then
 \begin{equation}\label{90}
 -\la -\tfrac2{\de^2} (s_\de-t_\de)
 + H\big(y_\de, -\tfrac 2{\de^2}\,d(x_\de,y_\de)\, \nabla_y d(x_\de,y_\de)
 - 2\de\, f(y_\de)\, \nabla_y f(y_\de)\big)\ge0.
 \end{equation}
 Subtracting \eqref{89}$-$\eqref{90} we get 
 \begin{align}\label{92a}
 2\la 
 + H\big(x_\de, P_\de\,+
   2 \de f(x_\de) \nabla f (x_\de)\big)
  -
  H\big(y_\de, Q_\de\, 
 - 2\de\, f(y_\de)\, \nabla_y f(y_\de)\big)\le 0,
 \end{align}
 where 
 \begin{equation}\label{PQde}
 P_\de := \tfrac 2{\de^2}d(x_\de,y_\de)\,\nabla_xd(x_\de,y_\de),
 \qquad
 Q_\de:=-\tfrac 2{\de^2} d(x_\de,y_\de) \,\nabla_y d(x_\de,y_\de)
 \end{equation}
 and $\la>0$.
 If we show that the last two terms in~\eqref{92a} are arbitrarily small we get 
 the desired contradiction with $\la>0$.

 By \eqref{dq}, $x_\de$ is near $y_\de$, and then 
 $|\nabla_x d(x_\de,y_\de)|=|\nabla_y d(x_\de,y_\de)|=1$.
 Therefore
 \begin{equation}\label{PQ}
 |P_\de|=|Q_\de|=\tfrac 2{\de^2} \,d(x_\de,y_\de).
 \end{equation}

 We need a shaper estimate for $\frac 1{\de^2} \,d(x_\de,y_\de)$.
 From \eqref{si}, \eqref{du} and \eqref{dq}  we get that 
 \begin{align}
 \de \big( f(x_\de)^2 + f(y_\de)^2\big)
 &\le K\,d(x_\de,y_\de)+ 2 KT \le 1+ 2KT,
 \notag\\
 \max\{ f(x_\de), f(y_\de)\}
 &\le \frac{Q_3(T)}{\sqrt{\de}} ,
 \label{93a}\\
 2\de\, f(x_\de) \,\lV \nabla f(x_\de)\rV 
 &\le 4\, Q_3(T)\, \sqrt{\de},
 \qquad\text{using~\eqref{nf}.}
 \label{94}
  \end{align}
  Since $\lV \nabla f\rV\le 2$, we have that $f$ has Lipschitz constant 2.
  Therefore
  \begin{align}
  |f(x_\de)-f(y_\de)| &\le 2 \, d(x_\de,y_\de)
  \le 2 Q_0\, \de 
  &\text{using~\eqref{dq}},
  \notag\\
  f(x_\de)+f(y_\de) &\le \frac{2 Q_3}{\sqrt{\de}} 
   &\text{from~\eqref{93a},}
   \notag\\
   |f(x_\de)^2-f(y_\de)^2| &\le 4\,Q_0\,Q_3\sqrt{\de}.
   \label{95}
   \end{align}

   We need a sharper estimate for $d(x_\de,y_\de)$.
  We have that 
  \begin{gather*}
  f(x_\de,t_\de,y_\de,s_\de) \ge f(x_\de,t_\de, x_\de,t_\de),
  \\
  v(x_\de,t_\de)-v(y_\de,s_\de) 
  +\la(t_\de-s_\de) -\tfrac 1{\de^2}\big(|t_\de-s_\de|^2+ d(x_\de,y_\de)^2\big)
  +\de\, \big(f(x_\de)^2-f(y_\de)^2\big)\ge 0.
  \end{gather*}
  \begin{align}
  \tfrac 1{\de^2} \big(|t_\de-s_\de|^2+d(x_\de,y_\de)^2\big)
  &\le v(x_\de,t_\de)-v(y_\de,s_\de) +\la |t_\de-s_\de|
  +\de \big(f(x_\de)^2-f(y_\de)^2\big)
  \notag\\
  &\le K\big( |t_\de-s_\de| +d(x_\de,y_\de)\big)
  + \la \,|t_\de-s_\de|
  + 4 Q_0 Q_3 \;\de^{\frac 32}
  \label{97}\\
  &\le Q_4\, \de
  \qquad\text{using \eqref{dq} and \eqref{dt},}
  \notag\\
  \max\{|t_\de-s_\de|,d(x_\de,y_\de)\} &\le Q_5\, \de^{\tfrac 32}.
  \label{98}
  \end{align}
  We  plug inequality \eqref{98} in the right hand side of inequality \eqref{97} and get
  \begin{equation}
  d(x_\de,y_\de) \le Q_6\; \de^{\frac 74}.
  \label{99}
  \end{equation}
 Then from \eqref{PQ} we get
 \begin{equation}
 |P_\de|=|Q_\de|\le Q_7\, \de^{-\frac 14}.
 \label{100}
 \end{equation}

  For $x,y\in \hM$ nearby, let $\tau_{xy}:T_x\hM\to T_y\hM$ 
  be the parallel transport along the minimal geodesic joining $x$ to $y$.
  Since  $H$ is  quadratic at infinity,
 there is $A>0$ such that
\begin{align}
|H_\e(x,p)-H_\e(y,\tau_{xy}(p))|&\le A \;(1+\tfrac 1{\e^2}|p|^2_x)\; d(x,y),
\label{101}
\\
|H_\e(x,p)-H_\e(x,q)| 
&\le A \,\tfrac 1{\e^2} \; (1+|p|_x+|q|_x) |p-q|_x.
\notag
\end{align}
Using \eqref{100} and \eqref{94}, we have that 
\begin{align}
\big| H_\e(x_\de,P_\de + 2 \de f(x_\de)\nabla f(x_\de))
-H_\e(x_\de, P_\de)\big|
&\le
A\tfrac 1{\e^2} (2+ 2|P_\de|)\,| 2\de f(x_\de)\nabla f(x_\de)|
\notag\\
&\le Q_8\, \de^{-\frac 14} \,\de^{\frac 12} = Q_8 \, \de^{\frac 14},
\label{102}\\
\lv H_\e(y_\de, Q_\de)-H_\e(y_\de,Q_\de -2 f(y_\de) \nabla f(y_\de))\rv
&\le Q_8\,\de^{\frac 14}.
\label{103}
\end{align}
From~\eqref{PQde}, \eqref{dq},
$Q_\de =\tau_{xy}(P_\de)$.
 Then using \eqref{101},~\eqref{100} and~\eqref{99},
\begin{align}
\lv H_\e(x_\de,P_\de) -H_\e(y_\de,Q_\de)\rv 
\le A(1+\tfrac 1{\e^2} (Q_7)^2\,\de^{-\frac 12}) \,Q_6\,\de^{\frac 74}\le Q_9\, \de^{\frac 54}.
\label{104}
\end{align}
Adding the inequalities \eqref{102}, \eqref{103}, \eqref{104} and comparing with 
\eqref{92a} we get
$$
2\la - Q_{10} \, \de^{\tfrac 14} \le 0,
$$
which is false for $\de$ small enough.

 \end{proof}

  \begin{Corollary}
  \label{Cuniqueness}
  \quad
  
  If $H:TM\to\re$ is a convex superlinear hamiltonian, $\hH$
  is its lift to $T\hM$, and $f:\hM\to \re$ is Lipschitz, then
  there is a unique Lipschitz viscosity solution 
  $u:\hM\times\re\to\re$ of the initial value problem
  \begin{equation}
  \begin{aligned}
  \partial_t u + \hH(x,\partial_x u) =0,
  \\
  u(x,0)= f(x).
  \end{aligned}
  \label{HJ3}
  \end{equation}
  \end{Corollary}
  
  \begin{proof}
  By Proposition \ref{eqlip}, the Lax formula,
  $$
  u(x,t) :=\min\,\Big\{\, f(\ga(0)) + \oint_\ga L \;:\; 
  \ga\in C^1\big(([0,T],T),(M,x)\big)\,\Big\}
  $$ 
  gives a Lipschitz viscosity solution of \eqref{HJ3}.
  Suppose that  $u$ and $v$ are Lipschitz viscosity solutions of \eqref{HJ3}.  
  Let 
  $$
  R_0 > \lV \partial_x u\rV_{\sup} + \lV\partial_x v\rV_{\sup}.
  $$
  There is $H_1$, quadratic at infinity such that $H_1(x,p)=H(x,p)$ if $|p|_x\le R_0$.
  For the lift $\hH_1:=H_1\circ (\pi^*)^{-1}$ we have that $u$ and $v$ are solutions of 
   \begin{equation}
  \begin{aligned}
  \partial_t u + \hH_1(x,\partial_x u) =0,
  \\
  u(x,0)= f(x).
  \end{aligned}
  \label{HJ4}
  \end{equation}
    Since $u$ and $v$ are Lipschitz solutions of \eqref{HJ4},
   they are both sub- and super- solutions of \eqref{HJ4}. 
   By Proposition~\ref{Pun},
  $u\le v$ and $v\le u$.
  
  \end{proof}

% \bibliographystyle{amsplain}
% \bibliography{biblio}

\begin{thebibliography}{1}

\bibitem{Bressan}
Alberto Bressan, \emph{Viscosity solutions of {H}amilton-{Jacobi} equations and
  optimal control problems}, Lecture notes available at
  http://www.math.psu.edu/bressan.

\bibitem{Concordel}
Marie~Christine Concordel, \emph{Periodic {H}omogenization of
  {H}amilton-{J}acobi equations}, Ph. D. Thesis, University of California at
  Berkeley, 1984.

\bibitem{mca2013}
Gonzalo Contreras, \emph{Homogenization on manifolds}, Mathematical {C}ongress
  of the {A}mericas, Contemp. Math., vol. 656, Amer. Math. Soc., Providence,
  RI, 2016, pp.~27--39.

\bibitem{CIPP}
Gonzalo Contreras, Renato Iturriaga, Gabriel~P. Paternain, and Miguel
  Paternain, \emph{Lagrangian graphs, minimizing measures and {M}a\~n\'e's
  critical values}, Geom. Funct. Anal. \textbf{8} (1998), no.~5, 788--809.

\bibitem{CIPP2}
\bysame, \emph{The {P}alais{-}{S}male condition and {Ma\~n\'e's} critical
  values}, Ann. Henri Poincar\'e (2000), 655--684.

\bibitem{CIS}
Gonzalo Contreras, Renato Iturriaga, and Antonio Siconolfi,
  \emph{Homogenization on arbitrary manifolds}, Calc. Var. Partial Differential
  Equations \textbf{52} (2015), no.~1-2, 237--252.

\bibitem{Ev5}
Lawrence~C. Evans, \emph{The perturbed test function method for viscosity
  solutions of nonlinear {PDE}}, Proc. Roy. Soc. Edinburgh Sect. A \textbf{111}
  (1989), no.~3-4, 359--375.

\bibitem{Mat5}
{John N.} Mather, \emph{Action minimizing invariant measures for positive
  definite {L}agrangian systems}, Math. Z. \textbf{207} (1991), no.~2,
  169--207.

\bibitem{Tran}
Hung~Vinh Tran, \emph{Hamilton-{J}acobi equations---theory and applications},
  Graduate Studies in Mathematics, vol. 213, American Mathematical Society,
  Providence, RI, [2021] \copyright 2021.

\end{thebibliography}

\def\cprime{$'$} \def\cprime{$'$} \def\cprime{$'$} \def\cprime{$'$}
\providecommand{\bysame}{\leavevmode\hbox to3em{\hrulefill}\thinspace}
\providecommand{\MR}{\relax\ifhmode\unskip\space\fi MR }
\providecommand{\MRhref}[2]{%
  \href{http://www.ams.org/mathscinet-getitem?mr=#1}{#2}
}
\providecommand{\href}[2]{#2}

\end{document}